\documentclass[11pt,reqno]{amsart}

\usepackage[T1]{fontenc}
\usepackage[utf8]{inputenc}
\usepackage{amsmath,amsfonts,amsthm,amssymb,amsxtra}
\usepackage{bbm} 
\usepackage[parfill]{parskip}    
\usepackage{enumerate}

\usepackage{comment}

\usepackage[ngerman, german, english]{babel} 

\newtheorem{theorem}{Theorem} [section]
\newtheorem{proposition}[theorem]{Proposition}
\newtheorem{lemma}[theorem]{Lemma}
\newtheorem{corollary}[theorem]{Corollary}

\theoremstyle{definition}

\newtheorem{definition}[theorem]{Definition}

\theoremstyle{remark}

\newtheorem{remark}[theorem]{Remark}
\newtheorem{remarks}[theorem]{Remarks}

\numberwithin{equation}{section}

\renewcommand{\epsilon}{\varepsilon}

\newcommand{\N}{\mathbb{N}}

\newcommand{\R}{\mathbb{R}}

\newcommand*\diff{\mathop{}\!\mathrm{d}} 
\newcommand{\eps}{\varepsilon}

\newcommand{\pzi}{\partial_{z_i}}
\newcommand{\pl}{\partial_{\lambda}}

\newcommand{\dhs}{{\dot{H}^s(\mathbb{R}^N)}}
\newcommand{\dhms}{{\dot{H}^{-s}(\mathbb{R}^N)}}

\newcommand{\Uzl}{{U[z,\lambda]}}
\newcommand{\lsc}{\left \langle}
\newcommand{\rsc}{\right \rangle}

\begin{document}

\title[Sharp extinction rates for fast diffusion equations]{Sharp extinction rates for positive solutions of fast diffusion equations}

\author[T.~König]{Tobias König}
\address[T.~König]{Goethe-Universität Frankfurt, Institut für Mathematik, 
Robert-Mayer-Str. 10, 60325 Frankfurt am Main, Germany.}
\email[T.~König]{koenig@mathematik.uni-frankfurt.de}

\author[M.~Yu]{Meng Yu}
\address[M.~Yu]{Goethe-Universität Frankfurt, Institut für Mathematik, 
Robert-Mayer-Str. 10, 60325 Frankfurt am Main, Germany.}
\email{yumeng161@mails.ucas.ac.cn}

\thanks{}

\begin{abstract}
Let $s \in (0, 1]$ and $N > 2s$. It is known that positive solutions to the (fractional) fast diffusion equation $\partial_t u + (-\Delta)^s (u^\frac{N-2s}{N+2s}) = 0$ on $(0, \infty) \times \R^N$ with regular enough initial datum  extinguish after some finite time $T_* > 0$. More precisely, one has $\frac{u(t,\cdot)}{U_{T_*, z, \lambda}(t,\cdot)} - 1 =o(1)$  as $t \to T_*^-$ for a certain extinction profile $U_{T_*, z, \lambda}$, uniformly on $\R^N$. In this paper, we prove the quantitative bound $ \frac{u(t,\cdot)}{U_{T_*, z, \lambda}(t,\cdot)} - 1 = \mathcal O(  (T_*-t)^\frac{N+2s}{N-2s+2})$, in a natural weighted energy norm. The main point here is that the exponent $\frac{N+2s}{N-2s+2}$ is sharp. This is the analogue of a recent result by Bonforte and Figalli (CPAM, 2021) valid for $s = 1$ and bounded domains $\Omega \subset \R^N$. Our result is new also in the local case $s = 1$. The main obstacle we overcome is the degeneracy of an associated linearized operator, which generically does not occur in the bounded domain setting.  

For a smooth bounded domain $\Omega \subset \R^N$, we prove similar results for positive solutions to $\partial_t u + (-\Delta)^s (u^m) = 0$ on $(0, \infty) \times \Omega$ with  Dirichlet boundary conditions when $s \in (0,1)$ and $m \in (\frac{N-2s}{N+2s}, 1)$, under a non-degeneracy assumption on the stationary solution. An important step here is to prove the convergence of the relative error, which is new for this case.
\end{abstract}

\maketitle

\section{Introduction and main result}

\subsection{Fast diffusion on $\R^N$}
For a given initial datum $u_0$, let $u$ be a solution to the \emph{(fractional) fast diffusion equation}
\begin{equation}\label{fde}
\begin{cases}
\partial_t u + (-\Delta)^s (|u|^{m-1}u) = 0, & t>0, \ x \in \R^N,  \\
u(0,x) = u_0(x), & x \in \R^N,
\end{cases}
\end{equation}
with {$s \in (0,1]$, $ N > 2s $ and}
\begin{align}\label{eq:m-fde}
    m = \frac{N-2s}{N+2s}.
\end{align}
It is proved in \cite{MR604277} for $s =1$ and in \cite{MR2954615} for $s \in (0,1)$ that given any initial datum $u_0 \geq 0$ with $ u_0 \in L^1(\R^N) \cap L^q(\R^N) $ for some $q > \frac{2N}{N+2s} $, there is a finite extinction time $T_* = T_*(u_0)$ and a unique strong solution $u$ such that $u(t,\cdot) > 0$ for every $t \in (0,T_*)$ and $u(t,\cdot) \equiv 0$ for $t \geq T_*$.

The main question we address here is to give a fine description of the extinction behavior of the solution $u(t,x)$ as $t \to T_*^-$. To explain this in more detail, given $T_* > 0$ we introduce a family of exact solutions to \eqref{fde} which has separated variables and extinguishes precisely at time $T_*$. Letting 
\begin{equation}
    \label{p definition RN}
    p := \frac{1}{m} = \frac{N+2s}{N-2s},
\end{equation}
this family is given, for parameters $z \in \R^N$ and $\lambda > 0$, by
\begin{align}\label{extinction-prof}
    U_{T_*,z,\lambda}(t,x):= \left(\frac{p-1}{p}\right)^{\frac{p}{p-1}}(T_* -t)^{\frac{p}{p-1}} U[z,\lambda]^{p}(x), \qquad t \in [0,T_*], \ x \in \R^N.
\end{align}
Here $U[z, \lambda]$ is the so-called \emph{Aubin--Talenti bubble function}, which is given by
\begin{equation}
\label{U z lambda}
U[z, \lambda](x) =  2^{\frac{N-2s}{2}}\left(\frac{\Gamma(\frac{N+2s}{2})}{\Gamma(\frac{N-2s}{2})}\right)^{\frac{N-2s}{4s}} \left( \frac{\lambda}{1 + \lambda^2|x-z|^2} \right)^\frac{N-2s}{2}, \qquad x \in \R^N.
\end{equation} 
Indeed, the choice of the constant in \eqref{U z lambda} ensures that 
\begin{equation}
    \label{U equation}
    (-\Delta)^s U[z, \lambda] = U[z, \lambda]^p \qquad \text{ on } \R^N
\end{equation} 
for all $z$ and $\lambda$. Using \eqref{U equation}, it can be checked by direct computation that the functions \eqref{extinction-prof} solve \eqref{fde}, for every $T_*$, $z$ and $\lambda$.

For $s = 1$, a fundamental result by Del Pino and Sáez \cite[Theorem 1.1]{MR1857048} guarantees that if 
\begin{equation}
    \label{condition delPino-Saez}
    0 \leq u_0 \in C(\R^N), \quad u_0 \not \equiv 0 \quad \text{ and } \, \sup_{x \in \R^N} (1 + |x|^{N+2}) u_0(x) < \infty
\end{equation}   and $u$ is the associated solution to \eqref{fde} with extinction time $T_* >0$, then there exist $z \in \R^N$ and $\lambda >0$ such that $u$ converges to $U_{T_*, z, \lambda}$ \emph{uniformly in relative error}. That is, 
\begin{align}\label{eq:asy}
\left\|\frac{u(t,\cdot)}{U_{T_*, z, \lambda}(t,\cdot)} - 1 \right\|_{L^\infty(\R^N)} \to 0 \qquad \text{ as } t \to T_*^-. 
\end{align}
The same conclusion is proved for $s \in (0,1)$ in \cite[Theorem 1.3]{Jin2014} under the slightly stronger assumption that
\begin{equation}
    \label{condition jin-xiong}
    \begin{aligned}
        & 0 < u_0 \in C^2(\R^N) \quad \text{ and  } \\
        & \text{$|x|^{2s-N} u^m_0\left( \frac{x}{|x|^2} \right)$ can be extended to a positive $C^2$-function in the origin.} 
    \end{aligned}
\end{equation}  
The main goal of our present work is to improve the qualitative convergence \eqref{eq:asy} to a \emph{quantitative} estimate of the form $\left\|\frac{u(t,\cdot)}{U_{T_*, z, \lambda}(t,\cdot)} - 1 \right\| \lesssim (T_* - t)^\kappa$ for some appropriate norm $\|\cdot\|$. In the estimate we will give, the exponent $\kappa = \kappa(N,s) > 0$ is  explicit, universal, and \emph{best possible}.

A major part of the subsequent analysis is most naturally and efficiently carried out after a change of time variable which we introduce now. 

Namely, for $u$ a nonnegative solution to \eqref{fde} and $p = \frac{1}{m} = \frac{N+2s}{N-2s}$, we set 
\begin{align}
\label{w in terms of u}
    w(\tau, x) :=\left(\frac{u(t,x)}{\left(\frac{p-1}{p}\right)^{\frac{p}{p-1}}(T_*-t)^{\frac{p}{p-1}} }\right)^{1/p} = \left(\frac{u(t,x)}{U_{T_*, z,\lambda}(t,x)}\right)^{1/p} U[z,\lambda](x),
\end{align}
where $t \in (0, T_*)$ and $\tau \in (0, \infty)$ are related by  
\begin{equation}
\label{change of variables t tau}
t = T_*\left(1-\exp\left(\frac{-(p-1)\tau}{p}\right)\right) \qquad \Leftrightarrow \qquad  T_*-t = T_*\exp\left(\frac{-(p-1)\tau}{p}\right). 
\end{equation} 

Straightforward computation shows that $w$ solves the equation 
\begin{align}\label{w equation}
\begin{cases}
\partial_\tau w^p + (-\Delta)^s w  = w^p, & \tau \in (0, \infty), \ x \in \R^N,  \\
w(0,x) = w_0(x), & x \in \R^N,
\end{cases}
\end{align}
with $w_0(x) = \left(\frac{p-1}{p}\right)^{-\frac{1}{p-1}} T_*^{-\frac{1}{p-1}} u_0(x)^{1/p} \geq 0$.

Through \eqref{w in terms of u}, 
the relative error convergence \eqref{eq:asy} is equivalent to 
\begin{align}\label{eq:asy w}
\left\|\frac{w(\tau,\cdot)}{U[z, \lambda]} - 1 \right\|_{L^\infty(\R^N)} \to 0 \qquad \text{ as } \tau \to \infty.
\end{align}

The following is our main theorem. It contains a quantitative bound on the distance between $u(t, \cdot)$ and $U_{T_*, z, \lambda}(t, \cdot)$, respectively between $w(\tau, \cdot)$ and $U[z, \lambda]$, measured in terms of the homogeneous Sobolev norm
\[ \|v\|_\dhs := \left( \int_{\R^N} |(-\Delta)^{s/2} v|^2 \diff x \right)^{1/2}. \]

\begin{theorem}
\label{theorem exponential decay}
Suppose that $u_0$ satisfies \eqref{condition delPino-Saez} if $s = 1$ and \eqref{condition jin-xiong} if $s \in (0,1)$. For the associated solution $u$ to \eqref{fde}, let $T_*> 0$, $z \in \R^N$ and $\lambda > 0$ be such that \eqref{eq:asy} holds. 

Then the solution $w$ to \eqref{w equation} associated to $u$ via \eqref{w in terms of u} and \eqref{change of variables t tau} satisfies
\begin{equation}
    \label{w bound theorem}
    \|w(\tau, \cdot) - U[z, \lambda]\|_{\dot H^s(\R^N)} \lesssim \exp \left(- \frac{4s}{N-2s+2} \tau \right) \qquad \text{ as } \tau \to \infty.  
\end{equation} 
Equivalently, 
\begin{equation}
    \label{u bound theorem}
    \left\|\left(\frac{u(t,\cdot)}{U_{T_*, z, \lambda}(t,\cdot)} - 1 \right) U[z,\lambda] \right\|_{\dhs} \lesssim  (T_*-t)^\frac{N+2s}{N-2s+2} \qquad \text{ as } t \to T_*^-.
\end{equation} 
\end{theorem}

The main point of this theorem is that the exponent $\frac{4s}{N-2s+2}$ in \eqref{w bound theorem}, and thus likewise the exponent $\frac{N+2s}{N-2s+2}$ in \eqref{u bound theorem}, is \emph{sharp} in the sense that it is given by the spectral gap of the linearization of \eqref{w equation} at the stationary solution $U[z, \lambda]$. A more detailed explanation of this sharpness can be found in Remark \ref{remark sharpness} below.

To our knowledge, it is only in the local case $s = 1$ and in cases different from ours that sharp extinction rates of solutions to the fast diffusion equation $\partial_t u = \Delta (u^m)$ (on either $\R^N$ or bounded domains $\Omega$) have so far been obtained:
\begin{itemize}
    \item In the important works \cite{MR2481073, BoDoGrVa}, sharp exponential convergence rates are deduced for solutions to $\partial_t u = \Delta (u^m)$ on $\R^N$, for general $m \in (0,1)$. However, this analysis is valid only for initial values $u_0$ which satisfy a rather restrictive trapping assumption. As observed in \cite{DaSe2008}, generic initial data fail to satisfy this assumption and present indeed a genuinely different extinction behavior. A more detailed discussion of this can be found in \cite[Section 1.3]{Ciraolo2018}. 
    \item For bounded domains $\Omega \subset \R^N$ with Dirichlet boundary conditions and subcritical diffusion exponent $m \in (\frac{N-2}{N+2},1)$, in the remarkable recent papers by Bonforte and Figalli \cite{Bonforte2021}, Akagi \cite{Akagi2023} and by Choi, McCann and Seis \cite{Choi2023} . 
\end{itemize}
 In particular, Theorem \ref{theorem exponential decay} is new also for the local case $s = 1$.

Our proof of Theorem \ref{theorem exponential decay} is mainly inspired by \cite{Akagi2023}. With respect to \cite{Akagi2023}, the decisive additional difficulty we face is that the problem on $\R^N$ is necessarily degenerate in the sense that the linearization of \eqref{w equation} at $U[z, \lambda]$ has a certain number of non-trivial solutions coming from the symmetries of the equation $(-\Delta)^s w = w^\frac{N+2s}{N-2s}$ on $\R^N$. These are generically absent in the case of a bounded domain treated in both \cite{Akagi2023} and \cite{Bonforte2021}. 

Dealing appropriately with these kernel elements is one of the main obstacles we overcome in this paper. For this, we exploit some ideas used in \cite{DeNitti2023} which are in turn inspired from \cite{Ciraolo2018}. The main technical innovations can be found in Lemmas \ref{lemma hard ineq} and \ref{lemma rho L rho estimate} below.

We close the discussion of Theorem \ref{theorem exponential decay} with some more remarks.

\begin{remarks}
\begin{enumerate}
[(a)]
\item For $s = 1$, the first exponential bound of the form \eqref{w bound theorem}, but with a non-explicit exponent, was obtained by Ciraolo, Figalli and Maggi in \cite{Ciraolo2018} (see also \cite{MR4090466}). In the recent paper \cite{DeNitti2023}, \eqref{w bound theorem} was proved for all $s \in (0,1]$, with an explicit, but non-sharp exponent.\footnote{We use this opportunity to point out that the definition of $U[z, \lambda]$ in \cite[eq. (1.4)]{DeNitti2023} is given with a wrong factor of $2^{2s}$. The factor $2^\frac{N-2s}{2}$ in \eqref{U z lambda} is the correct one. This can be conveniently checked, e.g., by following the computations in \cite[p. 13-14]{Frank2023}.}

\item As in \cite{Bonforte2021} and \cite{Akagi2023}, we prove convergence of $w$ towards $U[z,\lambda]$ with the  sharp exponent $\frac{4s}{N-2s+2}$  not in the strong uniform norm from \eqref{eq:asy w}, but in the weaker $\dot H^s$-norm. Through appropriate regularity theory, the bounds from Theorem \ref{theorem exponential decay} can subsequently be upgraded to yield quantitative convergence of the relative error. In this process,  sharpness of $\mu$ is however usually lost, compare e.g. \cite[Lemma 6.1]{Akagi2023} or \cite[eq. (1.16)]{Bonforte2021}. Indeed, by arguing as in  \cite[Proof of Theorem 2.8, Step 4]{DeNitti2023} (which relies on stereographic projection of \eqref{w equation} to the sphere and a Lipschitz regularity result from \cite{Jin2014}) we deduce from \eqref{w bound theorem} that 
\begin{equation}
    \label{w/U decay}
    \left \|\frac{w(\tau, \cdot)}{U[z, \lambda]} - 1 \right\|_{L^\infty(\R^N)} \lesssim \exp \left(-\frac{8s}{(N-2s+2)^2} \tau \right ), 
\end{equation} 
or equivalently
\[ \left \|\frac{u(t, \cdot)}{U_{T_*, z, \lambda}(t,\cdot)} - 1 \right\|_{L^\infty(\R^N)} \lesssim (T_* - t)^\frac{2(N+2s)}{(N-2s+2)^2}. \]

\item In the impressive recent paper \cite{Choi2023}, for $s = 1$ and bounded domains $\Omega$, the authors manage to overcome two of the difficulties discussed above, namely \emph{i)} dealing with a possibly non-trivial kernel of the linearized operator  and \emph{ii)} obtaining the sharp decay behavior in $L^\infty$-norm for the relative error function. The proof of these results relies on some quite involved smoothing estimates (see \cite[Section 5]{Choi2023}), which are not required for our proof of Theorem \ref{theorem exponential decay}. 
\end{enumerate}
\end{remarks}

\subsection{Fractional fast diffusion on a bounded domain}

We now turn to the fractional fast diffusion equation 
\begin{equation}\label{fde-bdd}
\begin{cases}
\partial_t u + (-\Delta)^s (|u|^{m-1}u) = 0, & t>0, \ x \in \Omega,  \\
u(0,x) = u_0(x), & x \in \Omega, 
\end{cases}
\end{equation}
on  a bounded domain $\Omega \subset \R^N $ with smooth boundary.\footnote{The majority of the statements we use are true if one only assumes $C^{2,\alpha}$ or even $C^{1,1}$ boundary regularity. But we do not strive for optimality in that respect and since some of the references we rely on, e.g. \cite{BoSiVa}, assume a smooth boundary, we do the same for simplicity.} Here, we assume $ s \in (0,1) $ and $ m \in ( \frac{N-2s}{N+2s} , 1 ) $.  The operator $(-\Delta)^s$ can be understood as any of the non-equivalent realizations (restricted, censored and spectral) of the fractional Dirichlet Laplacian on $\Omega$ which are studied most frequently. Rigorous definitions and some properties of these operators, including their domain $H(\Omega)$ and its dual $H^*(\Omega)$, are given in Section \ref{subsubsec laplacians}. Note in particular that in \eqref{fde-bdd} and the following equations like \eqref{w equation-bdd} and \eqref{stationary equation-bdd} the argument of $(-\Delta)^s$ is always understood to satisfy appropriate Dirichlet boundary conditions in the sense detailed in Section \ref{subsubsec laplacians}.

By \cite[Theorem 2.2]{BoSiVa} and \cite[Proposition 1]{BoIbIs}, we know that for any initial datum $u_0$ which satisfies 
\begin{equation}
    \label{u0 condition bounded}
    \begin{aligned}
    u_0 &\geq 0, \qquad u_0 \in H^*(\Omega), \\
       u_0 &\in L^1(\Omega) \quad \text{ if } \quad m \in \left(\frac{N-2s}{N}, 1\right), \quad \text{respectively} \\  
        \quad u_0 &\in L^q(\Omega) \quad  \text{for some $q > \frac{N(1-m)}{2s}$ } \quad \text{if } \quad m \in \left(\frac{N-2s}{N+2s}, \frac{N-2s}{N}\right],
    \end{aligned}
\end{equation}
there exists a unique solution $u \geq 0$ to \eqref{fde-bdd} which extinguishes after a finite time $T_* = T_*(u_0)$. 
(See Section \ref{subsubsec notions of sol} for a precise definition of solutions.)

For any solution $u = u(t,x) \geq 0$ to \eqref{fde-bdd}, we again define a function $w(\tau, x) \geq 0$ through 
\begin{align}
\label{w in terms of u bdd}
    w(\tau, x) :=\left(\frac{u(t,x)}{\left(\frac{p-1}{p}\right)^{\frac{p}{p-1}}(T_*-t)^{\frac{p}{p-1}} }\right)^{1/p},
\end{align}
for 
\begin{equation}
    \label{p definition bdd}
    p = \frac{1}{m} \in \left(1, \frac{N+2s}{N-2s}\right),
\end{equation}
where $t$ and $\tau$ are again related by \eqref{change of variables t tau}.  This function solves 
\begin{align}\label{w equation-bdd}
\begin{cases}
\partial_\tau ( w^p ) + (-\Delta)^s w  = w^p, & \tau \in (0, \infty), \ x \in \Omega,  \\
w(0,x) = w_0(x), & x \in \Omega,
\end{cases}
\end{align}
with $w_0(x) = \left(\frac{p-1}{p}\right)^{-\frac{1}{p-1}}(T_*-t)^{-\frac{1}{p-1}}  u_0(x) ^{\frac{1}{p}}$.

Assuming that a solution to \eqref{w equation-bdd} converges to a stationary state $\varphi \geq 0$, i.e., a nonnegative solution of 
\begin{align}\label{stationary equation-bdd}
(-\Delta)^s \varphi  = \varphi^p, & \ x \in \Omega,
\end{align}
we are again interested in the optimal convergence rate.

We say that a solution $\varphi \geq 0$ to \eqref{stationary equation-bdd} is \emph{non-degenerate} if the linearized operator 
\[ \mathcal L_\varphi := (-\Delta)^s - p \varphi^{p-1} \]
is invertible. By standard arguments (see Lemma \ref{lemma eigenvalues bounded}), there exists a basis of eigenfunctions $(e_j)_{j \in \mathbb N_0}$ and a sequence of eigenvalues $(\nu_j)_{\mathbb N_0}$ with $\nu_j \to \infty$ as $j \to \infty$ satisfying 
\[ \mathcal L_\varphi e_j = \nu_j  \varphi^{p-1} e_j.  \]
Thus $\varphi$ is non-degenerate if and only if none of the $\nu_j$ is equal to zero.  

For a solution $\varphi$ to \eqref{stationary equation-bdd}, we denote by $U_{T_*, \varphi} := \left(\frac{p-1}{p}\right)^{\frac{p}{p-1}}(T_*-t)^{\frac{p}{p-1}} \varphi(x)^p$ the solution of \eqref{fde-bdd} associated to $\varphi$ through \eqref{w in terms of u}. 

\begin{theorem}\label{thm-nonnegative-bdd}
    Let $p \in (1, \frac{N+2s}{N-2s})$ as defined in \eqref{p definition bdd}. Let $u_0$ satisfy \eqref{u0 condition bounded}, and let $u$ be the solution to \eqref{fde-bdd} with initial value $u_0$.  Let $ w = w( \tau , x ) $ be the solution to \eqref{w equation-bdd} associated to $u$ via \eqref{w in terms of u bdd} and let $\varphi$ be a positive solution to \eqref{stationary equation-bdd} such that $ w(\tau,\cdot) \to \varphi $ strongly in $H(\Omega)$ as $ \tau \to +\infty$. Assume that $ \varphi $ is non-degenerate. 
    
    Then 
    \begin{equation}
    \label{ineq thm bounded}
        0 \leq \| w(\tau) - \varphi\|_{H(\Omega)} \lesssim \exp\left(- \frac{\tilde{\nu}}{p} \tau\right) \qquad \text{ as }  \tau \to \infty, 
    \end{equation}
  where $\tilde \nu$ denotes the smallest positive eigenvalue of $\mathcal L_\varphi$.

  Equivalently, 
  \begin{equation}
    \label{ineq thm bounded u}
    \left\|\left(\frac{u(t,\cdot)}{U_{T_*, \varphi}(t,\cdot)} - 1 \right) \varphi \right\|_{H(\Omega)} \lesssim  (T_*-t)^\frac{\tilde \nu}{p-1} \qquad \text{ as } t \to T_*^-.
\end{equation} 
\end{theorem}

As in Theorem \ref{theorem exponential decay}, the exponent $\frac{\tilde \nu}{p}$ in \eqref{ineq thm bounded} (and thus likewise the exponent $\frac{\tilde \nu}{p-1}$ in \eqref{ineq thm bounded u}) is sharp with respect to the linearization about the limit solution $\varphi$, in the sense detailed in Remark \ref{remark sharpness}.

\begin{remarks}
\begin{enumerate}[(a)]
    \item For $s = 1$, Theorem \ref{thm-nonnegative-bdd} is proved in \cite{Bonforte2021, Akagi2023, Choi2023}. One may note that instead of the $H^1(\Omega)$ norm as in \ref{ineq thm bounded}, for technical reasons the authors in \cite{Bonforte2021} and \cite{Akagi2023} bound different energy quantities. Indeed, in \cite[Theorem 1.2]{Bonforte2021} the quantity $\int_\Omega |w - \varphi|^2 \varphi^{p-1}$ is employed, and the result in \cite[Theorem 1.4]{Akagi2023} uses the quantity $J(w(\tau, \cdot)) - J(\varphi)$, where $J$ is defined in \eqref{J definition} below (with $\R^N$ replaced by $\Omega$). These choices are however not essential,  since as observed in \cite[Corollaries 1.5 and 1.6]{Akagi2023} the sharp exponential convergence of $J(w(\tau, \cdot)) - J(\varphi)$ implies that of the other two quantities. (See also Lemma \ref{J-ctrl-rho} below.) Conversely, convergence in $H(\Omega)$ together with Sobolev's inequality implies convergence of $\int_\Omega |w - \varphi|^2 \varphi^{p-1}$ and of  $J(w(\tau, \cdot)) - J(\varphi)$. 

    \item Using Proposition \ref{proposition weighted smoothing effects-bdd}, Theorem \ref{thm-nonnegative-bdd} yields a quantitative decay rate for the relative error (with a non-sharp exponent), namely 
    \[ \left \| \frac{w(\tau, \cdot)}{\varphi} - 1 \right\|_{L^\infty(\Omega)} \lesssim   \exp \left( - \frac{\tilde \nu}{p} \frac{s}{N+ \gamma} \right) \qquad \text{ as } \tau \to \infty, \]
    or equivalently
    \[ \left \| \frac{u(t, \cdot)}{U_{T_*, \varphi}} - 1  \right\|_{L^\infty(\Omega)} \lesssim (T_*-t)^{\frac{\tilde \nu}{p-1} \frac{s}{N+\gamma}} \qquad \text{ as } t \to T_*^-,  \]
    where $\gamma \in (0,1]$ depends on the chosen realization of the fractional Dirichlet Laplacian; see  \eqref{gamma values}. 

    \item  In the case where $(-\Delta)^s$ is the restricted fractional Laplacian (see Section \ref{subsubsec laplacians}), it is proved in the recent preprint \cite{AkSa} that a nonnegative solution $w(\tau)$ to \eqref{w equation-bdd} always converges in $H(\Omega)$ to a single stationary state $\varphi$ solving \eqref{stationary equation-bdd}. This is the analogue of a fundamental result by Feireisl and Simondon \cite{Feireisl2000} for $s =1$. In this case, the assumption that $ w(\tau,\cdot) \to \varphi $ in $H(\Omega)$ can thus be removed from Theorem \ref{thm-nonnegative-bdd}. The proof from \cite{AkSa} appears to work also for other definitions of the fractional Dirichlet Laplacian. To emphasize the focus of our result and simplify its statement, we nevertheless choose to leave the convergence assumption of Theorem \ref{thm-nonnegative-bdd} in place.

\item In \cite{akagi2023optimal}, the authors prove optimal decay rates for \emph{sign-changing} solutions to fast diffusion equations for $s=1$ and  bounded domain $\Omega$. For $s \in (0,1)$, the same  remains an open problem to our knowledge. For sign-changing solutions to \eqref{fde} on all of $\R^N$, stating an analogue of Theorem \ref{theorem exponential decay} for sign-changing solutions is equally open for $s \in (0,1]$.
\end{enumerate}
\end{remarks}

An additional difficulty in the proof of Theorem \ref{thm-nonnegative-bdd} compared to that of Theorem \ref{theorem exponential decay} is that, to the best of our knowledge, no statement on the convergence as $\tau \to \infty$ of the relative error 
\begin{equation}\label{relative-error-def}
    \mathrm{h}(\tau, x)  := \frac{w(\tau,x)}{\varphi(x)} - 1
\end{equation}
is available in the literature for fractional $s \in (0,1)$ on bounded domains $\Omega$. 

Since the uniform convergence of $\mathrm h(\tau, \cdot)$ forms the starting point of our proof strategy for Theorem \ref{thm-nonnegative-bdd}, we prove the following. 

\begin{theorem}\label{theorem uniform_conv_rel_err}
   Let $p \in (1, \frac{N+2s}{N-2s})$. Let $u_0$ satisfy \eqref{u0 condition bounded}, and let $u$ be the solution to \eqref{fde-bdd} with initial value $u_0$.  Let $ w = w( \tau , x ) $ be the solution to \eqref{w equation-bdd} associated to $u$ via \eqref{w in terms of u bdd} and let $\varphi$ be a positive solution to \eqref{stationary equation-bdd} such that $ w(\tau,\cdot) \to \varphi $ strongly in $H(\Omega)$ as $ \tau \to +\infty$. Let $\mathrm{h}$ be the relative error defined in \eqref{relative-error-def}. Then,
    \begin{equation}
    \label{rel error conv thm}
        \lim_{\tau \to \infty} \| \mathrm{h}(\tau,\cdot) \|_{L^\infty (\Omega)} = 0.
    \end{equation}
\end{theorem}

\begin{remarks}
\begin{enumerate}[(a)]
     \item In the case $ s = 1 $, Theorem \ref{theorem uniform_conv_rel_err} has first been proved in  \cite[Theorem 2.1]{MR2863762} based on \cite{Feireisl2000} and the global Harnack principle developed by \cite[Proposition 6.2]{MR1119195}, where $\Omega$ is a bounded domain with $C^2$ boundary. Later on, it was extended to a quantitative convergence by \cite{Bonforte2021} with a proof independent of \cite{MR2863762} and only assuming $\Omega$ with $C^{1,1}$ boundary. 

    For $\Omega = \R^N$, $s \in (0,1)$ and under the assumption \eqref{condition jin-xiong}, Theorem \ref{theorem uniform_conv_rel_err} is proved in \cite[Theorem 1.3]{Jin2014}. 

    \item Our proof of Theorem \ref{theorem uniform_conv_rel_err} relies on some estimates analogous to those developed for $s= 1$ in \cite{Bonforte2021}, which control  $\|\mathrm{h(\tau, \cdot)}\|_\infty$ in terms of an $H^s$-type energy quantity.  The technical difficulty in carrying out the arguments from \cite{Bonforte2021} in our case is the following:  In \cite{Bonforte2021}, the convergence \eqref{rel error conv thm} appears already as an \emph{assumption} in the relevant estimate (because for $s=1$ it is known to hold thanks to the previous result \cite{MR2863762}). In order to obtain \eqref{rel error conv thm} as a \emph{conclusion} and hence prove Theorem \ref{theorem uniform_conv_rel_err}, we need a weaker kind of a-priori control on $\mathrm{h}$ to insert into the machinery from \cite{Bonforte2021}. We obtain such control by means of a global Harnack estimate in the spirit of \cite{BoVa2015} based on sharp Green's function estimates.  

\end{enumerate}
   
\end{remarks}

The outline of the rest of this paper is as follows. 

In Section 2, we discuss the eigenvalues of the linearized operator associated to \eqref{U equation} as a prerequisite for the proof of Theorem \ref{theorem exponential decay}. 

In Section 3, we give the main argument for the proof of Theorem \ref{theorem exponential decay}

In Section 4, we discuss the bounded domain case and prove Theorems \ref{thm-nonnegative-bdd} and \ref{theorem uniform_conv_rel_err}. 

\textit{Notation and definitions. } On smooth enough functions $u$ on $\R^N$, for $s \in (0,1)$ the fractional Laplacian $(-\Delta)^s$ acts as $(-\Delta)^s u(x) = c_{N,s} \lim_{\eps \to 0} \int_{\R^N \setminus B_\eps(x)} \frac{u(x) - u(y)}{|x-y|^{N+2s}} \diff y$ for some $c_{N,s} > 0$. We denote by $\dot H^s(\R^N)$ the space of functions $u \in L^1_\text{loc}(\R^N)$ such that 
\[ \|u\|_\dhs^2 := \int_{\R^N} |(-\Delta)^\frac{s}{2} u|^2 \diff x = \frac{c_{N,s}}{2} \iint_{\R^N \times \R^N} \frac{|u(x) - u(y)|^2}{|x-y|^{N+2s}} \diff x \diff y < \infty. \] 

For functions $f, g: I \to [0, \infty)$ for some time interval $I \subset [0, \infty)$, we write $f \lesssim g$ if there is  $C > 0$ such that $f(t) \leq C g(t)$ for all $t \in I$. Likewise, we write $f \gtrsim g$ if there is $c > 0$ such that $f(t) \geq c g(t)$ for all $t \in I$. The constant $c$ and $C$ appearing in the paper are understood to be independent of time (though they may depend on $N$, $s$, $\Omega$ and $u_0$) and allowed to change from line to line.  Finally we will often write  $u(t) = u(t, \cdot)$ and $w(\tau) = w(\tau, \cdot)$ to ease notation.

\section{The linearized operator and its eigenvalues}
\label{section linearized operator}

We begin by introducing the functional 
 \begin{equation}
 \label{J definition}
  J(v) := \frac{1}{2} \|v\|_\dhs^2 - \frac{1}{p+1} \int_{\R^N} |v|^{p+1} \diff x, 
 \end{equation}
which will play an important role in the proof of Theorem \ref{theorem exponential decay}.

Indeed, the main technical part of the proof of Theorem \ref{theorem exponential decay}, namely Lemma \ref{lemma hard ineq} below, relies on a Taylor expansion to second order of the functional $J(w(\tau))$ 
for $w(\tau)$ in the limit point $U[z, \lambda]$ (or in fact in a slight perturbation $U(\tau)$ of $U[z, \lambda]$). This is why the operator
\begin{equation}
    \label{L definition}
  \mathcal L_{U[z, \lambda]} =  D^2 J(U[z, \lambda]) = (-\Delta)^s - p U[z, \lambda]^{p-1}
\end{equation}
plays a central role in what follows.  

Here is the crucial information we will need about the eigenvalue problem 
\begin{equation}
    \label{L eigenvalue problem}
    \mathcal L_{U[z, \lambda]} e_j = \nu_j \Uzl^{p-1} e_j. 
\end{equation}

\begin{proposition}
    \label{proposition linearized operator}
    Let $\mathcal L_\Uzl$ be defined by \eqref{L definition}. Then there is an increasing sequence of eigenvalues $\nu_j \to \infty$ (counted with multiplicity) independent of $z, \lambda$ and and an associated sequence of eigenfunctions $(e_j)_{j \in \N_0} = (e_j[z, \lambda])_{j \in \N_0}$ such that \eqref{L eigenvalue problem} is satisfied for every $j$. Moreover, the following holds: 
    \begin{enumerate}[(i)]
        \item The $(e_j)_{j \in \N_0}$ form an orthonormal basis of $\dot{H}^s(\R^N)$ and an orthogonal basis of $L^2(\R^N, U[z, \lambda]^{p-1} \, \diff{x} )$.   
        \item We have $\nu_0 = 1 - p <0$, $\nu_1 = ... = \nu_{N+1} = 0$ and $\nu_{N+2} = p \frac{4s}{N-2s +2} > 0$. Moreover, 
        \begin{equation}
            \label{nondeg}
             \ker \mathcal L_\Uzl = \text{span} \left \{ \pzi \Uzl, \pl \Uzl \, : \, i = 1,...,N \right \}. 
        \end{equation} 
    \end{enumerate}
\end{proposition}

These properties are well known, but we provide references and some elements of proof nevertheless. 

The property \eqref{nondeg} is usually referred to as the \emph{non-degeneracy} of $\Uzl$ as a solution to \eqref{U equation}. It was first proved in 
\cite{Davila2013}. Differently from the setting on a bounded domain studied in \cite{Bonforte2021, Akagi2023}, here the notion of non-degeneracy allows for the kernel of $\mathcal L_\Uzl$ to be non-trivial. However, the crucial property is that the kernel consists only of elements which arise necessarily from the symmetries of equation \eqref{U equation}.

\begin{proof}
It suffices to prove the proposition for $U:= U[0, 1]$ because the general case follows from rescaling.  
For convenience let us introduce $B(x) = \left( \frac{2}{1 + |x|^2} \right)^{\frac{N-2s}{2}}$, which is just a differently normalized version of the bubble $U$, namely 
\begin{equation}
    \label{B to U}
    B = \alpha(0)^\frac{1}{1-p} U
\end{equation} 
for the number $\alpha(0)$ defined in \eqref{alpha ell} below. Define the auxiliary operator 
\[ \mathcal H_s := \frac{(-\Delta)^s}{B^{p-1}}\]
acting on the weighted space $L^2(\R^N, B^{p-1} \diff x)$. For a function $w \in \dhs$, define $v$ to be the function on $\mathbb S^N$ which satisfies
\begin{equation}
    \label{ster proj}
    w(x) = v(\mathcal S(x)) J_{\mathcal S}(x)^{\frac{1}{p+1}} = v(\mathcal S(x)) B(x),
\end{equation} 
where $\mathcal S: \R^N \to \mathbb S^N$ is (inverse) stereographic projection given by 
\begin{equation}
\label{ster proj definition}
(\mathcal S(x))_i = \frac{2 x_i}{1 + |x|^2} \quad (i = 1,...,N), \qquad (\mathcal S(x))_{N+1} = \frac{1 - |x|^2}{1+|x|^2}, 
\end{equation} 
and $J_{\mathcal S}(x) = |\det D \mathcal S(x)| = \left( \frac{2}{1 + |x|^2} \right)^N = B(x)^{p+1}$ is its Jacobian determinant. Then we have (see \cite[eq.(25)]{Frank2023}) 
\begin{equation}
    \label{conf trafo Hs norm}
    \lsc w, \mathcal H_s w \rsc_{L^2(\R^N, B^{p-1} \diff x)} =  \lsc w, (-\Delta)^s w \rsc_{L^2(\R^N)} = \lsc v, A_s v \rsc_{L^2(\mathbb S^N)}  
\end{equation} 
where $A_s$ is the operator given on spherical harmonics $Y_\ell$ of degree $\ell \geq 0$ of $\mathbb S^N$ as 
\[ A_s Y_\ell = \alpha(\ell) Y_\ell, \]
with
\begin{equation}
    \label{alpha ell}
    \alpha(\ell) = \frac{\Gamma(\ell + \frac{N}{2} + s)}{\Gamma(\ell + \frac{N}{2} - s)}. 
\end{equation}
Moreover, we directly see by change of variables that 
\begin{equation}
    \label{conf trafo L^2 norm}
    \lsc w, w \rsc_{L^2(\R^N, B^{p-1} \diff x)} = \int_{\R^N} w^2 B^{p-1} \diff x = \lsc v, v \rsc_{L^2(\mathbb S^N)}.  
\end{equation}
Thus $\mathcal H_s$ acting on $L^2(\R^N, B^{p-1} \diff x)$ is unitarily equivalent to $A_s$ via the map \eqref{ster proj}. In particular, the eigenvalues of $\mathcal H_s$ are precisely the $\alpha(\ell)$, and $\mathcal H_s \Tilde{e}_\ell = \alpha(\ell) \Tilde{e}_\ell$ for some $\Tilde{e}_\ell$ if and only if $\Tilde{e}_\ell$ is related to some spherical harmonic $Y_\ell$ of degree $\ell$ via \eqref{ster proj}. 

Recalling \eqref{B to U}, it follows that 
\begin{equation}
    \label{L eigenvalues proof}
    \mathcal L_U \Tilde{e}_\ell  = \left((-\Delta)^s - pU^{p-1} \right) \Tilde{e}_\ell = \left( \frac{\alpha(\ell)}{\alpha(0)} - p \right) U^{p-1} \Tilde{e}_\ell =: \Tilde{\nu}_\ell U^{p-1} \Tilde{e}_\ell.  
\end{equation} 
Now all the statements of the proposition can be read off by choosing an enumeration $\nu_j$ of the $\Tilde{\nu}_\ell$ which takes into account their respective multiplicities and corresponding eigenfunctions $(e_j)$ related via \eqref{ster proj} to pairwise orthogonal spherical harmonics. By \eqref{conf trafo Hs norm} and \eqref{conf trafo L^2 norm}, the $(e_j)$ are pairwise orthogonal in $\dhs$ and $L^2(\R^N, U^{p-1} \diff x)$ and we can normalize them to have unit $\dhs$ norm. This proves (i). 

For (ii), we simply notice that the space of spherical harmonics of degree $\ell = 0$ is one-dimensional and spanned by the constant function $1$, which gives via \eqref{ster proj} $e_0 = c U$ for some constant $c > 0$. Similarly, the space of spherical harmonics of degree $\ell = 1$ is spanned by the coordinate functions $\omega_i$ for $i=1,..., N+1$. A straightforward computation using \eqref{ster proj} now gives that, up to a positive constant, $e_i = \partial_{z_i} U$ and $e_{N+1} = \partial_\lambda U$ as claimed. The claimed values of $\nu_0$, $\nu_1$ and $\nu_{N+2}$ can be read off directly from \eqref{L eigenvalues proof} and \eqref{alpha ell}. This completes the proof. 
\end{proof}

\begin{remark}[Sharp exponential decay rate in terms of the eigenvalues of $\mathcal L_{U[z, \lambda]}$]
    \label{remark sharpness}
    We can now explain why we call the number $-\frac{8s}{N - 2s + 2}$ in Theorem \ref{theorem exponential decay} the sharp decay exponent and how this number is related to the eigenvalues of $\mathcal L_{U[z, \lambda]}$, more precisely to the smallest positive one $\mu_{N+2} = p \frac{4s}{N-2s +2}$. 

    Recall that ${U[z, \lambda]}$ is a stationary solution to \eqref{w equation}.     The linearization of \eqref{w equation} about $U:= U[z,\lambda]$ is 
\[ p U^{p-1} \partial_\tau \rho + (-\Delta)^s \rho = p U^{p-1} \rho \]
which is equivalent to 
\begin{equation}
    \label{h-equation}
    \partial_\tau \rho + (\mathsf L - 1)\rho = 0 
\end{equation} 
with $\mathsf L := \frac{(-\Delta)^s}{pU^{p-1}}$. A solution $w(\tau)$ to \eqref{w equation} converging to $U$ should thus asymptotically behave like $w \sim U + \rho$, where $\rho$ solves \eqref{h-equation} and tends to zero as $\tau \to \infty$.  

It is easily checked that the eigenvalues of $\mathsf L-1$ are $\kappa_j := \frac{\nu_j}{p}$ with associated eigenfunction $e_j$, for $\nu_j$ and $e_j$ as in Proposition \ref{proposition linearized operator}. Since \eqref{h-equation} is linear, its solutions can be explicitly expressed as 
\[ \rho(\tau, x) = \sum_{j \in \N_0} c_j e^{-\kappa_j \tau} e_j(x), \qquad c_j \in \R.  \]
Since in our situation we are interested in $\rho$ vanishing as $\tau \to \infty$, we consider $c_0 =...= c_{N+1} = 0$. The slowest possible exponential decay rate is given by $h$ of the form 
\[ \rho_0(\tau, x) = e^{-\kappa_{N+2} \tau} e_{N+2}(x). \]
Inserting $w(\tau) = U + \rho_0(\tau)$ into the functional $J$ and Taylor expanding as in Lemma \ref{Taylor_for} below,  we get 
\[ J(w(\tau)) - J(U) \sim \frac{1}{2} \lsc \rho_0(\tau), \mathcal L_U \rho_0(\tau) \rsc_{L^2(\R^N)} \sim e^{-2 \kappa_{N+2} \tau}.  \]
as the slowest possible decay rate. 
But $2 \kappa_{N+2} = 2 \frac{\nu_{N+2}}{p} = \frac{8s}{N-2s +2}$. In view of Lemma \ref{J-ctrl-rho}, the exponent given in Theorem \ref{theorem exponential decay} is the best possible one. 
\end{remark}

\section{Proof of Theorem \ref{theorem exponential decay}}
\label{section proof of theorem exp decay RN}

We let $w(\tau, x)$ be a solution to \eqref{w equation} and write $w(\tau) = w(\tau, \cdot)$. Recalling \eqref{eq:asy} and \eqref{eq:asy w} and replacing $w(\tau, x)$ by $\lambda^\frac{N-2s}{2}
 w(\tau, \lambda(x - z))$ for appropriate $z \in \R^N$ and $\lambda >0$ we will always assume in the following that 
\begin{equation}
\label{w/U to 0}
 \lim_{\tau \to \infty} \left\| \frac{w(\tau)}{U} - 1 \right\|_{L^\infty(\R^N)} = \lim_{t \to T_*^-} \left\| \left(\frac{u(t,x)}{U_{T_*, 0,1}(t,x)}\right)^{1/p} - 1 \right \|_{L^\infty(\R^N)} = 0,
\end{equation}
where $U := U[0,1]$.

By \cite[Theorem 5.1]{Jin2014}, respectively \cite[Proof of Proposition 5.1, Step 2]{MR1857048}, $w$ is strictly positive and smooth on $\R^N \times (0, \infty)$. This avoids in particular the need to discuss regularity issues.

The function $U(\tau)$ constructed in the following lemma will play a key role in resolving the difficulty that $\mathcal L_U$ has eigenvalue zero. 

\begin{lemma}
\label{lemma U(tau)}
  Let 
    \[ \mathcal M := \left\{ U[z, \lambda] \, : \, z \in \R^N, \, \lambda > 0 \right \} \, \subset \dhs. \]
    Then for $\tau$ large enough, there exists a unique function $U(\tau) = U[z(\tau), \lambda(\tau)]$  which satisfies 
\begin{equation}
    \label{U tau minimal condition}
     \|w(\tau) - U(\tau)\|_{\dhs} = \inf_{V \in \mathcal M} \|w(\tau) - V\|_\dhs. 
\end{equation}
\end{lemma}

\begin{proof}
    For the existence, see \cite[Lemma 3.4]{DeNitti2023}. The claimed uniqueness follows by the argument detailed in \cite[Proof of Lemma 12(d)]{Frank2023}. (For $s=1$, this lemma is due to \cite[Proposition 7]{Bahri1988}.)
\end{proof}

Here is another useful preliminary observation. 

\begin{lemma}
\label{J decreasing}
Let $w$ be a solution to \eqref{w equation}, let $J$ be given by \eqref{J definition} and let $p = \frac{N+2s}{N-2s}$. Then for every $\tau \in (0, \infty)$, we have
\[ \frac{\diff}{\diff  \tau} J(w(\tau)) = - \frac{4 p}{(p+1)^2} \int_{\R^N} |\partial_\tau (w(\tau)^\frac{p+1}{2})|^2 \, \diff{x} . \]
In particular, $J(w(\tau))$ is decreasing in $\tau \in (0, \infty)$. 
\end{lemma}

\begin{proof}
    Testing \eqref{w equation} by $ \partial_\tau w $ and using \eqref{w equation}, we get
    \begin{align*}
    p \int_{\R^N} w^{p-1} ( \partial_\tau w )^2 \diff x &= -\int_{\R^N} ( \partial_\tau w ) (-\Delta)^s w \, \diff{x}   + \int_{\R^N} ( \partial_\tau w ) w^p \, \diff{x} \\
    &= - \int_{\R^N} [(-\Delta)^{s/2} w] \, [(-\Delta)^{s/2} ( \partial_\tau w )] \, \diff{x} + \int_{\R^N} (\partial_\tau w) w^p \, \diff{x} \\
    &= - \langle J'(w), \partial_\tau w \rangle_* = - \frac{ \mathrm{d} }{ \mathrm{d}\tau } J(w(\tau)).
    \end{align*}
    Here, $\langle \cdot, \cdot \rangle_*$ denotes the dual pairing between $\dhs$ and its dual space $\dot H^{-s}(\R^N)$.
    Since
    $$
    \int_{\R^N} \left( \partial_\tau w^\frac{p+1}{2}  \right)^2 \, \diff{x} = \int_{\R^N} \left( \frac{p+1}{2} w^\frac{p-1}{2} \partial_\tau w \right)^2 \, \diff{x} = \frac{(p+1)^2}{4} \int_{\R^N} w^{p-1} ( \partial_\tau w )^2 \, \diff{x} ,
    $$
    the lemma follows. 
\end{proof}

Our strategy now consists in proving the following two inequalities involving the quantity $  \|J'(w(\tau))\|_{L^2(\R^N, U^{1-p} \, \diff{x})}^2$. 

\begin{lemma}
\label{lemma easy ineq}
Let $w$ be a solution to \eqref{w equation}, let $J$ be defined by \eqref{J definition} and let $p = \frac{N+2s}{N-2s}$. 
Then for every $\tau$, one has
\begin{equation}
    \label{easy ineq lemma}
    \|J'(w(\tau))\|_{L^2(\R^N, U^{1-p} \, \diff{x})}^2 \leq - p \left(1 + \left\|\frac{w(\tau)}{U}- 1\right\|_{L^\infty(\R^N)}\right)^{p-1} \frac{\diff}{\diff \tau} J(w (\tau)). 
\end{equation}  
\end{lemma}

\begin{proof}
By \eqref{J definition} and \eqref{w equation}, we have 
\[ J'(w(\tau)) = (-\Delta)^s w(\tau) - w(\tau)^p = - \partial_\tau (w(\tau)^p). \]
   Hence
    \begin{align*}
      \| J^{'} ( w(\tau) ) \|^2_{ L^2( \R^N , U^{1-p} \, \diff{x} ) } &= \int_{\R^N} ( \partial_\tau w^p (\tau) )^2 U^{1-p} \, \diff{x} \\
      &= \frac{ 4 p^2 }{( p + 1 )^2 } \int_{\R^N}  ( \partial_\tau w^\frac{p+1}{2} (\tau) )^2 \left( \frac{w}{U} \right)^{p-1} \, \diff{x} \\
      &\leq \frac{4p^2}{(p+1)^2} \left(1 + \left\|\frac{w(\tau)}{U}- 1\right\|_{L^\infty(\R^N)}\right)^{p-1}  \int_{\R^N}  ( \partial_\tau w^\frac{p+1}{2} (\tau) )^2 \diff x \\
      &= - p \left(1 + \left\|\frac{w(\tau)}{U}- 1\right\|_{L^\infty(\R^N)}\right)^{p-1}  \frac{\diff }{\diff \tau } J(w(\tau)),
    \end{align*}
    where the last equality follows from Lemma \ref{J decreasing}. 
\end{proof}

\begin{lemma}
\label{lemma hard ineq}
Let $w$ be a solution to \eqref{w equation}, let $J$ be given by \eqref{J definition} and let $p = \frac{N+2s}{N-2s}$. Let $U(\tau)$ be as in Lemma \ref{lemma U(tau)}. Then for every $\tau$ large enough, one has 
\begin{equation}
    \label{hard ineq lemma} 
    J(w(\tau)) - J(U) \leq  \left( \frac{1}{2 \nu_{N+2}} + \alpha(\tau) \right) \|J'(w(\tau))\|_{L^2(\R^N, U^{1-p} \, \diff{x})}^2.
\end{equation}
where 
\begin{equation}
    \label{alpha lemma hard}
    \alpha(\tau) = \mathcal O\left(\|w(\tau) - U\|_{\dhs}^\gamma + \left\|1 - \left(\frac{U}{U(\tau)}\right)^{p-1}\right\|_{L^\infty(\R^N)}\right). 
\end{equation}
Here, $\nu_{N+2} = p \frac{4s}{N-2s+2}$ is the eigenvalue from Proposition \ref{proposition linearized operator} and $\gamma \in (0,1]$ is the exponent from Lemma \ref{Taylor_for}.
\end{lemma}

Before proving Lemma \ref{lemma hard ineq}, which is the crucial and most involved part, let us show how the two inequalities from Lemmas \ref{lemma easy ineq} and \ref{lemma hard ineq} imply our main result.

\begin{proof}
[Proof of Theorem \ref{theorem exponential decay}]
We proceed in two steps. 

\textit{Step 1.  }
We first show that for every $\mu < \frac{8s}{N-2s+2}$ there is $C > 0$ such that for all large enough $\tau$, 
\begin{equation}
    \label{J convergence subcrit}
    J(w(\tau)) - J(U) \leq C \exp(- \mu \tau).
\end{equation} 
To see this, let $\eps > 0$, whose value will be chosen later. By  \cite[Proof of Theorem 2.8, Steps 1 and 2]{DeNitti2023}, \eqref{w/U to 0} implies that $w(\tau) \to U$ in $\dhs$ as $\tau \to \infty$. This implies $U(\tau) \to U$ in $\dhs$, so in particular the parameters of $U(\tau) = U[z(\tau), \lambda(\tau)]$ satisfy $(z(\tau), \lambda(\tau)) \to (0,1)$ by Lemma \ref{lemma U z lambda}(i). Together with \eqref{w/U to 0} and Lemma \ref{lemma U z lambda}(ii), we thus find 
\begin{equation}
    \label{qualitative bound}
    \left\|\frac{w(\tau)}{U}- 1\right\|_{L^\infty(\R^N)} + \|w(\tau)-U\|_\dhs^\gamma + \left\|1 - \left(\frac{U}{U(\tau)}\right)^{p-1}\right\|_{L^\infty(\R^N)} \leq \eps 
\end{equation}
for every $\tau$ large enough. Thus
  Lemma \ref{lemma easy ineq} and Lemma \ref{lemma hard ineq} combine to give the inequality
\begin{equation}
    \label{J epsilon estimate}
     J(w(\tau)) - J(U) \leq - p (1 + \eps)^{p-1} \left( \frac{1}{2 \nu_{N+2}} + \tilde C \eps \right)  \frac{\diff}{\diff \tau} J(w (\tau)) 
\end{equation}
for some $\tilde C > 0$, and every $\tau$ large enough. That is, 
\begin{equation}
    \label{diff ineq suboptimal}
    \frac{\frac{\diff}{\diff \tau} (J(w (\tau)) - J(U))}{J(w(\tau)) - J(U)} \leq - p^{-1} (1 + \eps)^{1-p}  \left( \frac{1}{2 \nu_{N+2}} + \tilde C \eps \right) ^{-1}. 
\end{equation} 
Now for any given $0 <\mu < \frac{2 \nu_{N+2}}{p} = \frac{8s}{N-2s+2}$, we can choose $\eps > 0$ such that 
\[ p^{-1} (1 + \eps)^{1-p}  \left( \frac{1}{2 \nu_{N+2}} +\tilde C \eps \right) ^{-1} = \mu. \]
By integrating the differential inequality \eqref{diff ineq suboptimal} from some fixed large $\tau_*$ to $\tau$, we obtain \eqref{J convergence subcrit}
with $C = (J(w(\tau_*) - J(U)) \exp (\mu \tau_*)$.  

\textit{Step 2.} We now complete the proof of Theorem \ref{theorem exponential decay} by showing that \eqref{J convergence subcrit} also holds for the limiting exponent, that is, 
\begin{equation}
    \label{J convergence crit}
     J(w(\tau)) - J(U) \leq C \exp \left(- \frac{8s}{N-2s+2} \tau \right).
\end{equation}
for every $\tau$ large enough. 

Thanks to Step 1, we have
\begin{equation}
    \label{exponential bound}
    \left\| \frac{w}{U} - 1 \right\|_{L^\infty (\R^N)} + \| w(\tau) - U \|_\dhs^\gamma +  \left\|1 - \left(\frac{U}{U(\tau)}\right)^{p-1}\right\|_{L^\infty(\R^N)} \leq C e^{-d \tau}
\end{equation}
for some $ d , C > 0 $ (which may both change from line to line in the following) and every $\tau$ large enough. Indeed, the exponential bound on the first two summands follows directly from \eqref{J convergence subcrit} by  Lemmas \ref{rel_err_exponential} and \ref{J-ctrl-rho}, respectively. 
Using the definition of $U(\tau)$ as $\dhs$-distance minimizer, we deduce moreover that
\[ \|U(\tau)-U\|_\dhs \leq \|w(\tau) - U(\tau)\|_\dhs + \|w(\tau) - U\|_\dhs \leq 2 \|w(\tau) - U\|_\dhs \leq C e^{-d\tau}. \] 
By Lemma \ref{lemma U z lambda}(i), this implies that $|z(\tau)| + |\lambda(\tau) - 1| \leq C e^{-d \tau}$. From this, by Lemma \ref{lemma U z lambda}(ii) we get
\[ \left\|1 - \left(\frac{U}{U(\tau)}\right)^{p-1}\right\|_{L^\infty(\R^N)}  \leq  C e^{-d \tau}, \]
which completes the proof of \eqref{exponential bound}. 

Repeating the argument in Step 1, but using \eqref{exponential bound} instead of \eqref{qualitative bound}, from Lemmas \ref{lemma easy ineq} and \ref{lemma hard ineq} we obtain \eqref{J epsilon estimate} and \eqref{diff ineq suboptimal}, but with $C e^{-d  \tau}$ instead of $\eps$. That is,
\begin{equation}
    \label{diff ineq optimal}
    \frac{\frac{\diff}{\diff \tau} (J(w (\tau)) - J(U))}{J(w(\tau)) - J(U)} \leq -  \frac{8s}{N-2s+2} + C e^{-d \tau} 
\end{equation} 
for every $\tau$ large enough. By integrating this inequality from some fixed large $ \tau_*$ to $\tau$, we get
$$
J(w(\tau)) - J(U) \leq \left( J(w(\tau_*)) - J(U) \right) e^{\frac{C}{d }} \exp \left(- \frac{8s}{N-2s+2} \left( \tau - \tau_* \right) \right),
$$
which proves \eqref{J convergence crit}.

By Lemma \ref{rel_err_exponential}, \eqref{J convergence crit} now directly implies \eqref{w bound theorem}. The bound \eqref{u bound theorem} follows  from \eqref{w bound theorem} via \eqref{w in terms of u} and \eqref{change of variables t tau}. 
\end{proof}

\subsection{The proof of Lemma \ref{lemma hard ineq}}

We now turn to the proof of the main technical estimate we have used, namely Lemma \ref{lemma hard ineq}. Our proof of this lemma is inspired by \cite[Section 5]{Akagi2023}. However, due to the fact that in our case the linearized operator $\mathcal L_{U(\tau)}$ has non-trivial kernel, we need to make several adjustments. This concerns most of all Step 1 below, which is not needed at all in \cite{Akagi2023}. In addition, we found it also necessary to modify the estimates in the main line of the argument in order to avoid the inverse operator $\mathcal L_U^{-1}$, which is made frequent use of in \cite{Akagi2023}, but which is ill-defined in our setting.   

\begin{proof}
    [Proof of Lemma \ref{lemma hard ineq}]

\textit{Step 1. Choice of Taylor basepoint.  } To treat the crucial additional difficulty (with respect to \cite{Bonforte2021, Akagi2023}) that $\ker \mathcal L_{U[z, \lambda]} \neq \{0\}$, we choose the basepoint for the Taylor expansion of $J$ \emph{depending on $\tau$}. Namely, we Taylor expand in the function $U(\tau)$ from Lemma \ref{lemma U(tau)} and write in the following 
\[ w(\tau) = U(\tau) + \rho(\tau). \]
Since the tangent space of $\mathcal M$ in $U(\tau)$ is precisely spanned by the functions $\pl U(\tau) = \frac{\diff}{\diff \lambda}|_{\lambda = \lambda(\tau)} U[z(\tau), \lambda]$ and $\pzi U(\tau) =  \frac{\diff}{\diff  z_i}|_{z = z(\tau)} U[z, \lambda(\tau)]$ for $i = 1,...,N$, the minimizing property of $U(\tau)$ implies that 
\begin{equation}
    \label{rho orthogonality}
    0 = \lsc \rho, \pl U(\tau)\rsc_\dhs = \lsc \rho, \pzi U(\tau)\rsc_\dhs \qquad \text{ for all } i=1,...,N. 
\end{equation}

\emph{Step 2. Expanding $J(w(\tau)) - J(U)$ in terms of eigenfunctions of $\mathcal L_{U(\tau)}$.  }

Noting that $J(U(\tau)) = J(U)$, by Lemma \ref{Taylor_for} we can expand 
\begin{align}
    J(w(\tau)) - J(U) &= J(w(\tau)) - J(U(\tau)) = \frac{1}{2} \lsc \rho(\tau), \mathcal L_{U(\tau)} \rho(\tau) \rsc_\ast + \mathcal O(\|\rho(\tau)\|_\dhs^{2 + \gamma}). \label{taylor}
\end{align}

By Proposition \ref{proposition linearized operator}, there is an orthonormal basis of $\dhs$ made from eigenfunctions $(e_j(\tau))_{j \in \N_0}$ of the operator $\mathcal L_{U(\tau)}$. For the following calculations, we now expand $\rho(\tau)$ into this basis, which gives
\[ \rho(\tau) = \sum_{j \in \N_0} \sigma_j(\tau) e_j(\tau) \]
for certain coefficients $\sigma_j(\tau) \in \R$. By Proposition \ref{proposition linearized operator}, the orthogonality condition \eqref{rho orthogonality} translates precisely to 
\begin{equation}
    \label{sigma j = 0}
     \sigma_j(\tau) = 0 \qquad \text{ for all } j = 1,...,N+1. 
\end{equation}

For the following estimates it will be important to keep track of the low eigenfunction $\sigma_0(\tau) e_0(\tau)$ separately. For this purpose, we introduce the notation 
\begin{equation}
    \label{Tilde rho}
    \Tilde{\rho}(\tau) := \sum_{j \geq N+2} \sigma_j(\tau) e_j(\tau), 
\end{equation} 
so that, in view of \eqref{sigma j = 0}, 
\begin{equation}
    \label{rho decomp}
    \rho(\tau) = \sigma_0(\tau) e_0(\tau) + \Tilde{\rho}(\tau). 
\end{equation}

Recalling that the eigenvalues $\nu_j$ introduced in Proposition \ref{proposition linearized operator} fulfill
\begin{equation}
    \label{nu j proof}
    \mathcal L_{U(\tau)} e_j  = \nu_j U(\tau)^{p-1} e_j, 
\end{equation} 
let us calculate the main term on the right side of \eqref{taylor}. It is convenient to introduce the numbers $\mu_j = \nu_j + p > 0$, so that  
\begin{equation}
    \label{mu j proof}
     (-\Delta)^s e_j(\tau) =  \mu_j U(\tau)^{p-1} e_j(\tau).  
\end{equation}
Since we choose the $(e_j(\tau))$ orthonormal in $\dot H^s(\R^N)$, this yields the relations
\begin{equation}
\label{eij products}
    \int_{\R^N} e_i(\tau) (-\Delta)^s e_j(\tau) = \delta_{ij}, \qquad  \int_{\R^N} e_i(\tau) \mathcal L_{U(\tau)} e_j(\tau)  = \delta_{ij} \frac{\nu_j}{\mu_j}, \qquad     \int_{\R^N} e_i(\tau) U(\tau)^{p-1} e_j(\tau) = \delta_{ij} \mu_j^{-1}
\end{equation}
which we will use repeatedly in the following computations.

From \eqref{taylor} we thus obtain 
\begin{align*}
        J(w(\tau)) - J(U(\tau)) &\leq \frac{1}{2} \sigma_0(\tau)^2 \frac{\nu_0}{\mu_0} + \frac{1}{2} \lsc \Tilde{\rho}(\tau), \mathcal L_{U(\tau)} \Tilde{\rho} (\tau) \rsc_\ast + \beta(\tau) \|\rho(\tau)\|_{\dhs}^2 \\
        &\leq \left( \frac{1}{2} \sigma_0(\tau)^2 \left(\frac{\nu_0}{\mu_0} + 2\beta(\tau)\right) \right) + \frac{1}{2} \lsc \Tilde{\rho}(\tau), \mathcal L_{U(\tau)} \Tilde{\rho}(\tau) \rsc_\ast + \beta(\tau) \|\Tilde{\rho}(\tau)\|_{\dhs}^2 \\
        &\leq \left( \frac{1}{2} \sigma_0(\tau)^2 \left(\frac{\nu_0}{\mu_0} + 2\beta(\tau)\right) \right) + \left( \frac{1}{2} + C\beta(\tau) \right) \lsc \Tilde{\rho}(\tau), \mathcal L_{U(\tau)} \Tilde{\rho}(\tau) \rsc_\ast,
\end{align*}
where $\beta(\tau) = \mathcal O(\|\rho(\tau)\|^\gamma_\dhs)$. The last inequality follows from the estimate 
\begin{equation}
    \label{rho tilde estimate}
    \|\Tilde{\rho}(\tau)\|_{\dhs}^2 = \sum_{j \geq N +2} \sigma_j(\tau)^2 \leq C \sum_{j \geq N +2} \frac{\nu_j}{\mu_j} \sigma_j(\tau)^2 = C \lsc \Tilde{\rho}(\tau), \mathcal L_{U(\tau)} \Tilde{\rho}(\tau) \rsc_\ast,
\end{equation}
which holds for some $C > 0$ because the sequence $(\frac{\nu_j}{\mu_j})_{j \geq N+2}$ is strictly positive and bounded away from zero. (Note that estimate \eqref{rho tilde estimate} would not necessarily be true for $\Tilde{\rho}(\tau)$ replaced by $\rho(\tau)$, because $\nu_0 < 0$.)

By Lemma \ref{lemma rho L rho estimate} below, this yields 
\begin{align*}
           & \qquad J(w(\tau)) - J(U(\tau)) \\
           &\leq  \frac{1}{2} \sigma_0(\tau)^2 \left(\frac{\nu_0}{\mu_0} + 2(\beta(\tau) + \delta(\tau))\right)  +  \left( \frac{1}{2} + C\beta(\tau) \right) \left( \frac{1}{\nu_{N+2}} + \delta(\tau) \right) \|J'(w(\tau))\|_{  L^2( \R^N , U(\tau)^{1-p} \diff{x} ) }^2 \\
           &\leq  \frac{1}{2} \sigma_0(\tau)^2 \left(\frac{\nu_0}{\mu_0} + 
 \mathcal O(\|\rho(\tau)\|_\dhs^\gamma) \right) + \left( \frac{1}{2 \nu_{N+2}} + \mathcal O(\|\rho(\tau)\|_\dhs^\gamma)  \right)\|J'(w(\tau))\|_{  L^2( \R^N , U(\tau)^{1-p} \diff{x} ) }^2 
\end{align*}
for $\tau$ large enough. Since  $\nu_0 = 1 - p < 0$, we have $\frac{\nu_0}{\mu_0} + \mathcal O(\|\rho(\tau)\|_\dhs^\gamma) < 0$ for $\tau$ large enough. Therefore, 
\begin{equation}
    \label{est prefinal}
    J(w(\tau)) - J(U(\tau)) \lesssim \left( \frac{1}{2 \nu_{N+2}} + \mathcal O(\|\rho(\tau)\|_\dhs^\gamma)  \right)\|J'(w(\tau))\|_{  L^2( \R^N , U(\tau)^{1-p} \diff{x} ) }^2. 
\end{equation} 

\textit{Step 3. Replacing $U(\tau)$ by $U$.}

Estimate \eqref{est prefinal} is already essentially of the desired form given in \eqref{hard ineq lemma}. However, for further use the statement  \eqref{hard ineq lemma}, where some occurrences of $U(\tau)$ in \eqref{est prefinal} are replaced by $U$, is more convenient to use, in particular in conjunction with Lemmas \ref{rel_err_exponential} and \ref{J-ctrl-rho} below. 

To turn \eqref{est prefinal} into \eqref{hard ineq lemma}, we first observe that by the minimality of $U(\tau)$, one has
\begin{equation}
    \label{est rho}
    \|\rho(\tau)\|_\dhs = \|w(\tau) - U(\tau)\|_{\dhs} \leq \|w(\tau) - U\|_{\dhs}. 
\end{equation} 
Moreover, 
\begin{align}
  \nonumber  &\qquad  \|J'(w(\tau))\|_{  L^2( \R^N , U(\tau)^{1-p} \diff{x} ) }^2 = \int_{\R^N} |J'(w(\tau))|^2 U(\tau)^{1-p} \diff x \\
  \nonumber  &= \int_{\R^N} |J'(w(\tau))|^2 U^{1-p} \left(1 - \left( 1 - \frac{U^{p-1}}{U(\tau)^{p-1}} \right) \right) \diff x \\
  \label{est J'}  &= \|J'(w(\tau))\|_{  L^2( \R^N , U^{1-p} \diff{x} ) }^2 \left( 1 + \mathcal O \left(   \left\|1 - \left(\frac{U}{U(\tau)}\right)^{p-1}\right\|_{L^\infty(\R^N)} \right) \right).
\end{align}
Combining \eqref{est prefinal}, \eqref{est rho} and \eqref{est J'}, we deduce \eqref{hard ineq lemma} with $\alpha(\tau)$ given by \eqref{alpha lemma hard}, and the proof is complete. 
\end{proof}

\begin{lemma}
    \label{lemma rho L rho estimate}
    Let $w$ solve \eqref{w equation}, let $J$ be given by \eqref{J definition} and let $p = \frac{N+2s}{N-2s}$. Let $\mathcal L_{U[z, \lambda]}$ be as in Proposition \ref{proposition linearized operator}, let $U(\tau)$ be as in Lemma \ref{lemma U(tau)} and let $\tilde{\rho}$ be defined by \eqref{Tilde rho}.
    
    Then for every $\tau$ large enough, we have
    \[ \lsc \Tilde{\rho}, \mathcal L_{U(\tau)} \Tilde{\rho} \rsc_\ast \leq \left( \frac{1}{\nu_{N+2}} + \delta(\tau) \right) \|J'(w(\tau))\|_{  L^2( \R^N , U(\tau)^{1-p} \diff{x} ) }^2  + \delta(\tau) \sigma_0(\tau)^2,   \]
    where $\delta(\tau) = \mathcal O(\|w(\tau) - U\|^\gamma_\dhs)$. 
\end{lemma}

\begin{proof}
By the Taylor expansion from Lemma \ref{Taylor_for}, we can write 
\[ J'(w(\tau)) = \mathcal L_{U(\tau)} \rho(\tau) + e(\tau), \]
where $e(\tau) \in \dhms$ is an error term such that $\|e(\tau)\|_{\dot{H}^{-s}(\R^N)} = \mathcal O(\|\rho(\tau)\|_\dhs^{1 + \gamma})$ as $\tau \to \infty$, for some $\gamma \in (0, 1]$. 
Moreover, by orthogonality we can write $\lsc \Tilde{\rho}(\tau), \mathcal L_{U(\tau)} \Tilde{\rho}(\tau) \rsc_\ast = \lsc \Tilde{\rho}(\tau), \mathcal L_{U(\tau)} \rho(\tau) \rsc_\ast$. Thus
\begin{align*}
    \lsc \Tilde{\rho}(\tau), \mathcal L_{U(\tau)} \Tilde{\rho}(\tau) \rsc_\ast =  \lsc \Tilde{\rho}(\tau), J'(w(\tau)) \rsc_\ast -  \lsc \Tilde{\rho}(\tau), e(\tau) \rsc_\ast. 
\end{align*}
The error term can be estimated by 
\begin{align*}
     \left| \lsc \Tilde{\rho}(\tau), e(\tau) \rsc_\ast \right| &\leq \|\Tilde{\rho}(\tau)\|_{\dhs} \|e(\tau)\|_{\dhms} \leq \delta(\tau) \|\Tilde{\rho}(\tau)\|_{\dhs} \|{\rho}(\tau)\|_{\dhs} \\
     &\leq \delta(\tau) \sigma_0(\tau)^2 + \delta(\tau) \|\Tilde{\rho}(\tau)\|_{\dhs}^2
\end{align*}
for every $\tau$ large enough, where $\delta(\tau) = \mathcal O(\|\rho(\tau)\|_\dhs)$. Using \eqref{rho tilde estimate}, we get 
\begin{equation}
    \label{1 - C intermediate}
    (1 - C \delta(\tau)) \lsc \Tilde{\rho}(\tau), \mathcal L_{U(\tau)} \Tilde{\rho}(\tau) \rsc_\ast \leq \lsc \Tilde{\rho}(\tau), J'(w(\tau)) \rsc_\ast + \delta(\tau) \sigma_0(\tau)^2.    
\end{equation} 
Now by \eqref{eij products} and Cauchy-Schwarz, we have
\begin{align*}
    \lsc \Tilde{\rho}(\tau), J'(w(\tau)) \rsc_\ast &= \int_{\R^N} \Tilde{\rho}(\tau) J'(w(\tau)) \\
    &\leq \left( \int_{\R^N} \Tilde{\rho}(\tau)^2 U(\tau)^{p-1} \right)^{1/2} \|J'(w(\tau))\|_{  L^2( \R^N , U(\tau)^{1-p} \diff{x} ) } \\
    &= \left( \sum_{j \geq N+2} \sigma_j(\tau)^2 \mu_j^{-1}  \right)^{1/2} \|J'(w(\tau))\|_{  L^2( \R^N , U(\tau)^{1-p} \diff{x} ) }  \\
    &\leq \left( \sum_{j \geq N+2} \sigma_j(\tau)^2 \frac{\nu_j}{\mu_j} \right)^{1/2} \frac{1}{\sqrt{\nu_{N+2}}}\|J'(w(\tau))\|_{  L^2( \R^N , U(\tau)^{1-p} \diff{x} ) } \\
    &\leq \frac{1}{2}  \sum_{j \geq N+2} \sigma_j(\tau)^2 \frac{\nu_j}{\mu_j} + \frac{1}{2 \nu_{N+2}} \|J'(w(\tau))\|_{  L^2( \R^N , U(\tau)^{1-p} \diff{x} ) }^2 \\
    &= \frac{1}{2} \lsc \Tilde{\rho}(\tau), \mathcal L_{U(\tau)} \Tilde{\rho}(\tau) \rsc_\ast +  \frac{1}{2 \nu_{N+2}} \|J'(w(\tau))\|_{  L^2( \R^N , U(\tau)^{1-p} \diff{x} ) }^2. 
\end{align*}  
Inserting this estimate into \eqref{1 - C intermediate}, we obtain 
\[  \left(\frac{1}{2} - C \delta(\tau) \right) \lsc \Tilde{\rho}(\tau), \mathcal L_{U(\tau)} \Tilde{\rho}(\tau) \rsc_\ast \leq \frac{1}{2 \nu_{N+2}} \|J'(w(\tau))\|_{  L^2( \R^N , U(\tau)^{1-p} \diff{x} ) }^2 + \delta(\tau) \sigma_0(\tau)^2.  \]
Dividing by $\left(\frac{1}{2} - C \delta(\tau) \right)$ gives the   conclusion. 
\end{proof}

\subsection{Auxiliary lemmas}
\label{subsec auxiliary}

In this section, we state and prove several auxiliary lemmas we previously used to prove Theorem \ref{theorem exponential decay}.

The following lemma gives a fine enough Taylor expansion of $J$ and its derivative. We apply it several times in the proof of Theorem \ref{theorem exponential decay}.

\begin{lemma}\label{Taylor_for}
For $\tau$ large enough, let $U(\tau)$ be as in Lemma \ref{lemma U(tau)} and let $\mathcal L_{U(\tau)}$ be defined by \eqref{L definition}.
Then
    \begin{equation}\label{Taylor_1}
        J(w(\tau)) - J(U(\tau)) = \frac{1}{2} \lsc w(\tau) - U(\tau), \mathcal L_{U(\tau)} ( w(\tau) - U(\tau) ) \rsc_\ast + E(\tau),
    \end{equation}
    \begin{equation}\label{Taylor_2}
        J'(w(\tau)) = \mathcal L_{U(\tau)} ( w(\tau) - U(\tau) ) + e(\tau).
    \end{equation}
Here, $\lsc \cdot, \cdot \rsc_\ast$ denotes the dual pairing between $\dhs$ and $\dot{H}^{-s}(\R^N)$, $E(\tau) \in \R$ and $ e(\tau) \in \dhms $ denote generic terms which satisfy
\begin{equation}\label{err_term}
   \lim_{\tau \to \infty} \frac{|E(\tau)|}{\| w(\tau) - U(\tau) \|_\dhs^{2+\gamma} } < +\infty \quad \text{and} \quad \lim_{\tau \to \infty} \frac{ \| e(\tau) \|_\dhms }{ \| w(\tau) - U(\tau) \|_\dhs^{1+\gamma} } < +\infty 
\end{equation}
for some $ \gamma \in (0,1] $.

The same expansions hold for $U(\tau)$ replaced by $U = \lim_{\tau \to \infty} w(\tau)$. 
\end{lemma}

\begin{proof}
    When $ p \geq 2 $, $J$ is $C^3$ in $\dhs$. 
    Using Taylor's theorem in general Banach spaces (see \cite[Theorem A.2]{Akagi2023}) and recalling $ J^{'}(U(\tau)) = 0 $, we can immediately prove \eqref{Taylor_1} and \eqref{Taylor_2} with $ \gamma = 1$. 
    
    When $ 1 < p < 2 $, we choose 
    \[ X_1 : = \left\{ w \in \dhs : w U(\tau)^\frac{p-2}{2} \in L^{2 \cdot 2_* } (\R^N) \right\},  \] 
    where $ 2_* : = \frac{2N}{N+2s} $ equipped with the norm $ \| w \|_{X_1}^2 : = \| w \|_\dhs^2 + \| w U(\tau)^\frac{p-2}{2} \|_{L^{2 \cdot 2_* } (\R^N)}^2 $. Moreover, we let 
    \[ X_2 := \left\{ w \in \dhs : w U(\tau)^\frac{p-2}{3} \in L^3(\R^N) \right\} \] 
    be equipped with the norm $ \| w \|_{X_2}^3 := \| w \|_\dhs^3 + \| w U(\tau)^\frac{p-2}{3} \|_{L^3(\R^N)}^3 $. Arguing as in \cite[Section 7]{Akagi2023}, we can show $J^{''} : X_1 \to \mathcal{L}( X_1 , \dhms )$ is Gateaux differentiable at $ U_\theta(\tau) := (1-\theta) U(\tau) + \theta w(\tau) $ for $ \tau > \tau_1 $ and the Gateaux derivative of $J^{''}$ at $\phi_\theta$ is bounded in $\mathcal{L}^{(2)} ( X_1 , \dhms )$ for $ \theta \in [0,1]$. Similarly, $J^{''} : X_2 \to \mathcal{L}^{(2)}( X_2 , \R )$ is Gateaux differentiable at $U_\theta(\tau)$ in $X_2$ for any $ \theta \in [0,1]$ and its Gateaux derivative is bounded in $ \mathcal{L}^{(3)} (X_2,\R)$ for $ \theta \in [0,1]$. As in \cite[Section 7]{Akagi2023}, it follows that \eqref{err_term} holds with $\gamma = p-1$.  

    The proof for $U$ instead of $U(\tau)$ is identical. 
\end{proof}

Next, we state and prove two norm estimates needed in the proof of Theorem \ref{theorem exponential decay}.

The first of these estimates is a regularity result which allows to control the relative error through (a worse power of) the $\dhs$-distance. 

\begin{lemma}\label{rel_err_exponential}
Let $w$ solve \eqref{w equation} and let $U = U[0,1]$ be given by \eqref{U z lambda}. Then for every $\tau > 0$ large enough, we have 
\[
   \left\| \frac{w(\tau)}{U} - 1 \right\|_{L^\infty (\R^N)} \lesssim \| w(\tau) - U \|_\dhs^\frac{2}{N-2s+2}.
\]
\end{lemma}

This lemma is contained in \cite[Proof of Theorem 2.8, Step 4]{DeNitti2023}, but we sketch the proof for completeness.

\begin{proof}
Defining $v$ as in \eqref{ster proj} through stereographic projection \eqref{ster proj definition}, we have 
\[ \left\|\frac{w(\tau)}{U} - 1 \right\|_{L^\infty(\R^N)} = \alpha(0)^\frac{1}{1-p}  \left\|v(\tau) - \alpha(0)^\frac{1}{p-1} \right\|_{L^\infty(\mathbb S^N)}. \]
 Moreover, $v(\tau)$ is uniformly Lipschitz on $\mathbb S^N$ by \cite[eq. (4.2) and Proposition 5.1]{MR1857048} (for $s =1$), respectively \cite[Proof of Theorem 3.1 and Proposition 5.2]{Jin2014} (for $s \in (0,1)$). 
Since $\|\frac{w(\tau)}{U} - 1 \|_{L^\infty(\R^N)} \to 0$ as $\tau \to \infty$ by \eqref{w/U to 0}, the interpolation estimate \cite[Lemma 4.2]{DeNitti2023} is applicable and yields 
\[ \|v(\tau) - \alpha(0)^\frac{1}{p-1} \|_{L^\infty(\mathbb S^N)} \lesssim \|v(\tau) - \alpha(0)^\frac{1}{p-1} \|_{L^{p+1}(\mathbb S^N)}^\frac{2}{N-2s+2}. \]
Transforming back to $\R^N$ via \eqref{ster proj} and using Sobolev's inequality yields the lemma. 
\end{proof}

The second estimate permits to control the $\dhs$-distance, for which Theorem \ref{theorem exponential decay} is stated, in terms of the difference $J(w(\tau)) - J(U)$ which is being used in most of the proof.

\begin{lemma}\label{J-ctrl-rho}
Let $w$ solve \eqref{w equation}, let $J$ be given by \eqref{J definition} and let $U = U[0,1]$ be given by \eqref{U z lambda}. Then for every $\tau$ large enough, we have
\[ \|w(\tau) - U\|_\dhs \lesssim \left( J(w(\tau)) - J(U) \right)^\frac{1}{2}.\]
\end{lemma}

\begin{proof}[Proof of Lemma \ref{J-ctrl-rho}]
    By the Taylor expansion from Lemma \ref{Taylor_for} (used here with basepoint $U$ instead of $U(\tau)$),
\begin{align}
    &\qquad J(w(\tau)) - J(U)  \nonumber \\
    &=  \frac{1}{2} \|w(\tau) - U \|^2_\dhs - \frac{p}{2} \int_\R U^{p-1} | w(\tau) - U |^2 ~\diff{x} + o(\|w(\tau) - U\|_\dhs^2)\nonumber \\
    &\geq \frac{1}{2} \| w(\tau) - U \|^2_\dhs - \frac{p}{2} \| w^p(\tau) - U^p \|^2_{  L^2( \R^N , U^{1-p} \diff{x} ) }+ o(\|w(\tau) - U\|_\dhs^2). \label{exp-ctrl-rho}
 \end{align} 
For the last estimate, we used the elementary inequality 
$$
| a - b | \leq a^{1-p} | a^p - b^p | \quad \text{for any $ a , b > 0 $ and $ p  \geq 1$,}
$$
with $a = U$ and $b = w(\tau)$.

By Lemmas \ref{lemma easy ineq} and \ref{lemma hard ineq}, we see that
\begin{align*}
    \left( J(w(\tau)) - J(U) \right)^\frac{1}{2} \| J^{'} (w(\tau)) \|_{  L^2( \R^N , U^{1-p} \diff{x} ) } &\leq C \| J^{'} (w(\tau)) \|_{  L^2( \R^N , U^{1-p} \diff{x} ) }^2 \\
    &\leq - C \frac{\mathrm{d}}{\mathrm{d}\tau} \left( J(w(\tau)) - J(U) \right),
\end{align*}
which yields
$$
\| J^{'} (w(\tau)) \|_{  L^2( \R^N , U^{1-p} \diff{x} ) } \leq - C \frac{\mathrm{d}}{\mathrm{d}\tau} \left( J(w(\tau)) - J(U) \right)^{1/2}.
$$
Recalling that $ \partial_\sigma (w^p) (\sigma) = - J^{'} (w(\sigma)) $, this gives
\begin{equation*}
    \begin{aligned}
        \| w^p(\tau) - U^p \|_{  L^2( \R^N , U^{1-p} \diff{x} ) } &\leq \int_\tau^\infty \| \partial_\sigma ( w^p )(\sigma) \|_{  L^2( \R^N , U^{1-p} \diff{x} ) } ~\mathrm{d}\sigma \\
        &\leq C \left(J(w(\tau)) - J(U) \right)^\frac{1}{2}.
    \end{aligned}
\end{equation*}
The lemma now follows from inserting this bound into \eqref{exp-ctrl-rho} and using that \\ $\lim_{\tau \to \infty}\|\rho(\tau)\|_\dhs=0$ by \eqref{w/U to 0} together with \cite[Proof of Theorem 2.8, Steps 1 and 2]{DeNitti2023}.
\end{proof}

The last lemma we state in this section contains some useful estimates concerning the family $U[z, \lambda]$ as $(z, \lambda) \to (0,1)$. These estimates are both elementary and expected. Since we did not find a suitable reference, we provide a short proof.

\begin{lemma}
    \label{lemma U z lambda}
    For $z \in \R^N$ and $\lambda > 0$, let $U[z, \lambda]$ be given by \eqref{U z lambda} and let $U:= U[0,1]$. As $z \to 0$ and $\lambda \to 1$, the following holds:
    \begin{enumerate}[(i)]
        \item $\|U[z, \lambda] - U\|_\dhs = (c+o(1)) (|z|^2 + |\lambda-1|^2)^\frac{1}{2} $, for some $c = c(N,s) > 0$. 
        \item $\left\|1 - \frac{U[z, \lambda]}{U}\right\|_{L^\infty(\R^N)} = \mathcal O(|\lambda-1| + |z|)$.     
    \end{enumerate}
\end{lemma}

\begin{proof}       
    To prove \emph{(i)}, using equation \eqref{U equation} we write 
            \begin{equation}
                \label{U Hs norm}
                 \|U[z, \lambda] - U\|_\dhs^2 = 2 \|U\|^2_\dhs - 2 \int_{\R^N} U[z, \lambda]^p U \diff x,
            \end{equation}
                where $p = \frac{N+2s}{N-2s}$. Direct computation shows that the function 
    \[ F: \R^N \times (0, \infty) \to \R, \qquad (z, \lambda) \mapsto \int_{\R^N} U[z, \lambda]^p U \diff x \]
    satisfies 
    \[ F(0,1) = \int_{\R^N} U^{p+1} \diff x = \|U\|_\dhs^2, \quad \nabla_{(z, \lambda)} F(0,1) = 0, \quad D^2_{(z, \lambda)} F(0,1) = d \text{ id}, \]
    for some $d = d(N,s) > 0$; see \cite[Proof of Lemma 12.(d)]{Frank2023} for an equivalent computation on $\mathbb S^N$. From \eqref{U Hs norm}, a Taylor expansion of $F(z, \lambda)$ in $(0,1)$ gives the claim with $c = \sqrt d$. 

    To prove \emph{(ii)}, it is enough to show that 
    \begin{equation}
        \label{(ii) goal}
        \sup_{x \in \R^N} |F_{z, \lambda}(x)|  \lesssim |\lambda - 1| + |z|, \quad \text{ where }  \quad F_{z, \lambda}(x) = \frac{1 + \lambda^2 |x-z|^2}{\lambda( 1+ |x|^2)} - 1. 
    \end{equation} 
   Since $\lim_{|x| \to \infty} F_{z, \lambda}(x) = \lambda - 1$, we only need to show that every critical point $y \in \R^N$ of $F_{z, \lambda}$ verifies $F_{z, \lambda}(y) \lesssim |\lambda - 1| + |z|$. 

   The  critical point equation $\nabla F_{z,\lambda} (y) = 0$ is equivalent to 
    \begin{equation}
        \label{y crit point}
         (y - z) (1 + |y|^2) = y (\lambda^{-2} + |y-z|^2). 
    \end{equation}
    In particular, $z = ry$ for some $r = r(z, \lambda, y) \in \R$. Inserting \eqref{y crit point} back into $F_{z, \lambda}$ gives 
    \[ F_{z, \lambda}(y) = \lambda(1 - r) - 1 = (\lambda - 1) - \lambda r, \]
    so to conclude \eqref{(ii) goal} it remains to show that 
    \begin{equation}
        \label{r estimate}
         |r| \lesssim |\lambda - 1| + |z|. 
    \end{equation}
    To prove \eqref{r estimate} we may clearly assume $r \neq 0$. Inserting $y = \frac{1}{r} z$ into \eqref{y crit point}, taking the absolute value and rearranging terms, we arrive at 
    \begin{equation}
        \label{r equation}
        r^3 = r^2 (1-\lambda^{-2}) + (r - r^2) |z|^2. 
    \end{equation} 
    Clearly, $|r| \to \infty$ is impossible because the right side of \eqref{r equation} is a quadratic polynomial with bounded coefficients as $z \to 0$ and $\lambda \to 1$. Now from \eqref{r equation}, the fact that $r$ is bounded gives directly 
    \[ |r| \lesssim |\lambda - 1|^\frac{1}{3} + |z|^\frac{2}{3}.  \]
    Reinserting this once again into \eqref{r equation} easily yields \eqref{r estimate}. Thus \emph{(ii)} is proved. 
        \end{proof}

\section{Bounded domain case}

In this section we prove Theorem \ref{thm-nonnegative-bdd}. We follow the same strategy as in the proof of Theorem \ref{theorem exponential decay}. The most substantial part of this section will be devoted to proving uniform convergence of the relative error $\frac{w(\tau, \cdot)}{\varphi} - 1$, that is, Theorem \ref{theorem uniform_conv_rel_err}.

\subsection{Preliminaries}

\subsubsection{Different realizations of the fractional Dirichlet Laplacian}
\label{subsubsec laplacians}
We begin by discussing the various definitions of the fractional Laplacian $(-\Delta)^s$ with Dirichlet boundary conditions on $\Omega$. Our Theorems \ref{thm-nonnegative-bdd} and \ref{theorem uniform_conv_rel_err} hold for any of the following non-equivalent realizations of $(-\Delta)^s$ on $\Omega$. 
\begin{itemize}
    \item The \emph{restricted fractional Laplacian (RFL)} $(-\Delta)^s_\text{re}$, whose quadratic form is given by 
    \[ q_\text{re}(u) = \iint_{\R^N \times \R^N} \frac{(u(x) - u(y))^2}{|x-y|^{N+2s}} \diff x \diff y, \]
    for 
    \[ u \in H^s_\text{re}(\Omega) := \left\{ u \in H^s(\R^N) \, : \, u \equiv 0 \, \text{ on } \R^N \setminus \Omega \right\}. \]
    \item The \emph{censored fractional Laplacian (CFL)} $(-\Delta)^s_\text{ce}$, whose quadratic form is given by
    \[ q_\text{ce}(u) = \iint_{\Omega \times \Omega} \frac{(u(x) - u(y))^2}{|x-y|^{N+2s}} \diff x \diff y 
 \]
 for 
 \[ u \in H^s_\text{ce}(\Omega) := \left\{ u \in L^2(\Omega) \, : \, \iint_{\Omega \times \Omega} \frac{(u(x) - u(y))^2}{|x-y|^{N+2s}} \diff x \diff y  < \infty, \quad u = 0 \text{ on } \partial \Omega \right\}. \]
 For this operator, the Dirichlet boundary condition only makes sense when $s \in (\frac{1}{2}, 1)$, and we will always assume this condition implicitly in statements about $(-\Delta)^s_\text{ce}$. The equality $u = 0$ on $\partial \Omega$ is intended in the sense of traces, which is well-defined if $s > \frac{1}{2}$ (see, e.g., \cite[Proposition 4.5]{MR2944369}).
 
 \item The \textit{spectral fractional Laplacian (SFL)} $(-\Delta)^s_\text{sp}$ defined through the Dirichlet eigenvalues $(\lambda_k)_{k \in \N} \subset (0, \infty)$ of $-\Delta$ on $\Omega$ and the associated eigenfunctions $(\varphi_k)_{k \in \N} \subset H^1(\Omega)$ by the quadratic form
 \[ q_\text{sp}(u) := \sum_{k=1}^\infty \lambda_k^s |u_k|^2. \]
 for every $u = \sum_{k=1}^\infty u_k \varphi_k$ such that 
 \[ u \in H^s_\text{sp}(\Omega) := \left\{ u \in L^2(\Omega) \, :\, \sum_{k=1}^\infty \lambda_k^s |u_k|^2 < \infty \right\}. 
 \]
\end{itemize}
We refer to \cite[Section 2.2]{BoIbIs} for a more exhaustive discussion of these operators.

Here, we only record one property which we will crucially use in the following. Namely, each of the operators $(-\Delta)^s_\#$ admits a Green's function $\mathbb G_\#: \Omega \times \Omega \to \R$ satisfying the two-sided estimate \cite[eq.s (K3) and (K4)]{BoIbIs}
\begin{equation}
    \label{Greens estimate}
    \Phi_\#(x) \Phi_\#(y) \lesssim \mathbb G_\#(x,y) \lesssim \frac{1}{|x-y|^{N-2s}} \left( \frac{\Phi_\#(x)}{|x-y|^{\gamma_\#}} \land 1 \right) \left( \frac{\Phi_\#(y)}{|x-y|^{\gamma_\#}} \land 1 \right). 
\end{equation}
The placeholder $\#$ here is intended to stand for any of the subscripts $\{\text{re}, \text{ce}, \text{sp} \}$. The symbol $\Phi_\#$ denotes the lowest eigenfunction of the operator $(-\Delta)^s_\#$ with eigenvalue $\lambda_{1, \#} > 0$.  We always take $\Phi_\#$ to be nonnegative and, for definiteness, normalized to satisfy
$\|\Phi_\#\|_{L^2(\Omega)} = 1$. The values of the exponent $\gamma_\#$ is given by 
\begin{equation}
    \label{gamma values}
    \gamma_\text{re} = s, \qquad \gamma_\text{ce} = 2s-1 \, \text{ (when $s \in (1/2, 1)$)}, \qquad \gamma_\text{sp} = 1. 
\end{equation} 

The analysis of this section can probably be extended without great difficulty to more general diffusion operators whose Green's function satisfies \eqref{Greens estimate}, as done in \cite{BoIbIs}. In our discussion, we however restrict to the important case of $(-\Delta)^s_\#$ for concreteness.

For ease of notation, we shall in the following drop the subscript $\#$ whenever all cases are treated in a unified fashion and no confusion can arise. That is, we simply denote $(-\Delta)^s_\# = (-\Delta)^s$, $q_\# = q$, $\mathbb G_\# = \mathbb G$, $\Phi_\# = \Phi$, $\lambda_{1, \#} = \lambda_1$ and $\gamma_\# = \gamma$. 

To avoid confusion with the standard Sobolev spaces $H^s(\Omega)$ and following the notation from \cite{BoSiVa, BoIbIs}, we denote $H^s_\#(\Omega) = H(\Omega)$ and write $H^*(\Omega)$ for its dual space.

\subsubsection{Notions of solution}
\label{subsubsec notions of sol}

We recall here the notion of weak solution defined in \cite[Definition 4.1]{MR2954615} for problem \eqref{fde-bdd}.
\begin{definition} \label{weak-sol-def}
    A function $u$ is a \emph{weak solution} for problem \eqref{fde-bdd} if:
    \begin{itemize}
        \item $u \in C([0,\infty) : L^1(\Omega)) $ and $|u|^{m-1} u \in L^2_{loc}( ( 0,\infty ) : H(\Omega)) $;
        \item 
        $\int_0^\infty \int_\Omega u \frac{\partial \varphi}{\partial t} \diff x\diff t - \int_0^\infty q(|u|^{m-1} u, \varphi) \diff t = 0$
         for every $\varphi \in C_0^1 ( (0,\infty) \times \Omega)$ (where $q$ is the quadratic form associated to $(-\Delta)^s$);
        \item $ u(\cdot,0) = u_0 $ almost everywhere in $\Omega$.
    \end{itemize}
    A weak solution is called a \emph{strong solution} if $\partial_t u \in L^\infty ( (\tau,\infty) : L^1(\Omega))$ for every $\tau > 0$.
\end{definition}The existence of a weak solution to \eqref{fde-bdd} for SFL and RFL is proved in \cite[Theorem 2.2]{BoSiVa}. In this paper, the authors defined the concept of weak dual solution, expressed in terms of the problem involving the inverse of the fractional Laplacian, and work in this abstract framework. Moreover, according to the remarks after \cite[Proposition 3.1]{BoSiVa}, the weak dual solution obtained in \cite[Theorem 2.2]{BoSiVa} by using the celebrated theory of Brezis \cite{Brezis_1971,MR348562} and Komura \cite{MR216342} on maximal monotone operators is moreover unique and in fact is also a bounded weak solution in the sense of Definition \ref{weak-sol-def}. Furthermore, by comparison principle, $u_0 \geq 0$ implies $u(t) \geq 0$ for all $t \geq 0$ and moreover, this nonnegative solution is also strong in the sense of Definition \ref{weak-sol-def}. In the recent paper \cite{BoIbIs} these results and observations have been extended to the CFL and more general operators; see in particular \cite[Theorems 2.4 and 2.5]{BoIbIs}.

\subsection{A global Harnack principle}

This section serves as a preparation for the proof of Theorem \ref{theorem uniform_conv_rel_err}. We establish the following global Harnack principle.
It extends \cite[Theorem 8.1]{BoVa2015}, which is valid for diffusion exponents $m>1$, to the fast diffusion range $m < 1$. Recall that $ \Phi \in H(\Omega) $ is the first eigenfunction of $(-\Delta)^s$ on $\Omega$ with eigenvalue $\lambda_1 > 0$.

\begin{proposition}[Global Harnack Principle]\label{GHP}
    For $ m \in (\frac{N-2s}{N+2s},1)$, let $u$ be a weak solution to \eqref{fde-bdd} with inital value $u_0$ satisfying \eqref{u0 condition bounded} and extinction time $ T_* = T_*(u_0) > 0 $. 
    Then there exist a time $ t_* \in (0, T_*)$ and universal constants $C_0, C_1 > 0$ such that
    \begin{equation}\label{GHP-ine}
    C_0 (T_*-t)^\frac{m}{1-m} \leq \frac{u^m(t,x)}{\Phi(x)} \leq C_1 (T_*-t)^\frac{m}{1-m}
    \end{equation}
    for all $ t \in (t_*,T_*] $ and $x \in \Omega$. 
\end{proposition}

By applying the change of variables \eqref{w in terms of u}, we deduce the following estimate for solutions $w$ to \eqref{w equation-bdd}. 

\begin{corollary}[Global Harnack Principle for $w$]
    \label{corollary harnack w}
    For $ m \in (\frac{N-2s}{N+2s},1)$, let $u$ be a weak solution to \eqref{fde-bdd} with inital value $u_0$ satisfying \eqref{u0 condition bounded} and extinction time $ T_* = T_*(u_0) > 0 $. 
Let $w$ be the solution of \eqref{w equation-bdd} associated to $u$ through \eqref{w in terms of u bdd} and \eqref{change of variables t tau}. Let $\varphi$ be any solution to \eqref{stationary equation-bdd}. Then there exists a time $\tau_* > 0$ and constants $\tilde{C_0}, \tilde{C_1} > 0$ such that 
    \begin{equation}
        \label{w/varphi harnack}
        \tilde{C_0} \leq \frac{w(\tau, x)}{\varphi(x)} \leq \tilde{C_1} 
    \end{equation} 
for every $\tau \in [\tau_*, \infty)$ and every $x \in \Omega$. 
\end{corollary}

\begin{proof}
    The estimate \eqref{w/varphi harnack}, but for the ratio $\frac{w}{\Phi}$ instead of $\frac{w}{\varphi}$, follows directly from \eqref{GHP-ine} and \eqref{w in terms of u bdd}. By applying once again Proposition \ref{GHP} to the function $u$ associated to the stationary solution $w(\tau, x) = \varphi(x)$, it also follows that also $\frac{\varphi}{\Phi}$ is bounded from above and below by positive constants. Combining these two estimates, \eqref{w/varphi harnack} follows. 
\end{proof}

Our proof follows the arguments from \cite{BoVa2015}, which we adapt to $m < 1$ by using some estimates proved recently in \cite{BoIbIs}.

\begin{proof}[Proof of Proposition \ref{GHP}]

\textit{Step 1: Preliminaries. }
By fundamental pointwise estimates proved in \cite[Lemma 3.4]{BoIbIs}, we have that for all $0 \leq t_0 \leq t_1$,
\begin{equation}\label{fund-point-inequ}
\begin{aligned}
    u^m(t_1,x) &\leq \frac{t_1^\frac{m}{1-m}}{1-m} \int_\Omega \frac{u(t_0,y) - u(t_1,y)}{t_1^\frac{1}{1-m} - t_0^\frac{1}{1-m}} \mathbb G(x,y) \diff{y} \\
    &\leq \frac{1}{1-m} \frac{t_1^\frac{m}{1-m}}{t_1^\frac{1}{1-m} - t_0^\frac{1}{1-m}} \| u(t_0) \|_{L^\infty(\Omega)} \|\mathbb{G}(x,\cdot)\|_{L^1(\Omega)}. 
\end{aligned}
\end{equation}
Indeed, the statement from \cite[Lemma 3.4]{BoIbIs} is valid for bounded and nonnegative gradient flow solutions. The weak solutions we consider belong to this category by the discussion following Definition \ref{weak-sol-def}. 

For all $ t \geq \frac{T_*}{2} $ we choose $ t_0 = t - \frac{T_*-t}{2} \in [\frac{T_*}{4}, T_*]$ and $t_1 = t $. With this choice, $t - t_0 = \frac{T_*-t}{2} $ and $ T_*-t_0 = 3 \frac{T_*-t}{2} $.
Using $u(t_1, \cdot) \geq 0$, we get
\begin{equation}\label{u-m-upper-bound-est-prelim}
u^m(t,x) \leq C (T_*-t)^{-1} \int_{\Omega} u(t_0, y) \mathbb G(x,y)  \diff y \qquad \text{ for all } t \in [\frac{T_*}{2}, T_*].
\end{equation}

On the other hand, using the smoothing effects \cite[Theorem 2.6]{BoIbIs} together with the decay of the $L^{m+1}$ norm \cite[Proposition 4]{BoIbIs} we obtain, for all $t_0 \in [\frac{T_*}{4},T_*]$,
    \begin{equation}\label{L-infty-est}
    \| u(t_0) \|_{L^\infty(\Omega)} \leq C (T_*-t_0)^\frac{1}{1-m}.
    \end{equation}

Moreover, from \eqref{Greens estimate} it can be deduced (see, e.g., \cite[Lemma 3.3]{BoIbIs}) that
\begin{equation}
    \label{Green-L1-est}
    \|\mathbb{G}(x,\cdot)\|_{L^1(\Omega)} \leq C \mathcal{B}_1(\Phi(x))
\end{equation}
with
\begin{equation}
\label{B1Phi}
    \mathcal{B}_1(\Phi(x)) = 
    \left\{
    \begin{aligned}
        &\Phi(x), \quad \text{for any}~ 2s > 1, \\
        &\Phi(x) |\log \Phi(x)|, \quad \text{for} ~ 2s = 1,\\
        &\Phi(x)^{2s}, \quad \text{for any}~ 2s < 1.
    \end{aligned}
    \right.
\end{equation}
Inserting \eqref{L-infty-est} and \eqref{Green-L1-est} into \eqref{u-m-upper-bound-est-prelim} yields
\begin{equation}
    \label{u-m-upper-bound-est}
    u^{m}(t,x) \leq C \mathcal B_1(\Phi(x)) (T_* - t)^\frac{m}{1-m} \qquad \text{ for all } t \geq \frac{T_*}{2}.
\end{equation}

\textit{Step 2. The upper bound. }

We can now prove the upper bound from \eqref{GHP-ine}. Because of \eqref{B1Phi}, we need to treat three different cases depending on the value of $s$.

\underline{Case (i): $s > \frac{1}{2}$.} In this case, $\mathcal B_1(\Phi(x)) = \Phi(x)$, and so \eqref{u-m-upper-bound-est} is already  the desired upper bound on $\frac{u^m}{\Phi}$ from \eqref{GHP-ine}. 

\underline{Case (ii): $s < \frac{1}{2}$.} In this case, $\mathcal B_1(\Phi(x)) = \Phi(x)^{2s}$, so \eqref{u-m-upper-bound-est} is weaker than \eqref{GHP-ine}. However, we are able to improve \eqref{u-m-upper-bound-est} through an iterative procedure inspired by the proof of \cite[Lemma 11.2]{BoVa2015}.

Indeed, \eqref{u-m-upper-bound-est} reads, in this case,
\begin{equation}\label{u-upper-bound}
    u(t,x)  \leq C (T_*-t)^\frac{1}{1-m} \Phi^{\mu_1} (x) \qquad \text{ for } t \in [\frac{T_*}{2},T_*] ,
\end{equation}
with $\mu_1 = \frac{2s}{m}$. 

If $\mu_1 \geq 1$, then we insert \eqref{u-upper-bound} with $t_0$ back into \eqref{u-m-upper-bound-est-prelim}, to find 
\begin{equation}
    \label{mu geq 1}
    u^m(t,x) \leq C (T_*-t)^\frac{m}{1-m} \int_\Omega \Phi(y) \mathbb G(x,y) \diff y =  C \lambda_1^{-1} (T_*-t)^\frac{m}{1-m} \Phi(x). 
\end{equation} 
 The application of \eqref{u-upper-bound} with $t_0$ is valid here for every $t$ such that $t_0 = t - \frac{T_* - t}{2} \geq \frac{T_*}{2}$. So we have proved \eqref{GHP-ine} for every $t \geq \frac{2}{3} T_*$. 

 If $\mu_1 < 1$, we can still use the same idea to iteratively improve the exponent $\mu_1$ in \eqref{u-upper-bound}. To achieve this, we insert again the bound \eqref{u-upper-bound} at time $t_0$ into \eqref{u-m-upper-bound-est-prelim}. Applying Hölder's inequality and \eqref{Green-L1-est} then gives
$$
\begin{aligned}
u^m(t,x) &\leq C (T_*-t)^\frac{1}{1-m} \int_\Omega \Phi^{\mu_1}(y) \mathbb G(x,y) \diff{y} \\
&\leq C (T_*-t)^\frac{1}{1-m} \left( \int_\Omega \mathbb G (x,y) \diff{y} \right)^{1-\mu_1} \left( \int_\Omega \Phi(y) \mathbb G (x,y) \diff{y} \right)^{\mu_1} \\
&\leq C (T_*-t)^\frac{1}{1-m} \Phi(x)^{2s(1-\mu_1) + \mu_1}.
\end{aligned}
$$
That is, we have improved \eqref{u-upper-bound} to 
$$
u(t,x) \leq C (T_*-t)^\frac{1}{1-m}   \Phi(x)^{\mu_2} \qquad \text{ for all } t \geq \frac{2}{3} T_*
$$
with $\mu_2 = \frac{2s(1 - \mu_1) + \mu_1}{m} \geq \mu_1$. Iterating this process for as long as $\mu_k < 1$ yields a sequence $(\mu_n)$, defined recursively by $\mu_{n+1} = \frac{2s(1 - \mu_n) + \mu_n}{m}$ such that
\[ 
u(t,x) \leq C (T_*-t)^\frac{1}{1-m}  \Phi(x)^{\mu_n} \qquad \text{ for all } t \in [q_n T_*, T_*]. 
\]
Here, the sequence $(q_n)$ is defined recursively by $q_1 = \frac{1}{2}$ and $q_{n+1} = \frac{2}{3}(q_n + \frac{1}{2})$. It is easy to see that $(q_n)$ is strictly increasing with $\lim_{n \to \infty} q_n = 1$. Since the ratio $\frac{\mu_{n+1}}{\mu_n}$ is bounded from below by $\frac{1}{m} > 1$, after a finite number $K-1$ of iteration steps we obtain $\mu_K \geq 1$. Repeating now once more the estimate from \eqref{mu geq 1} gives \eqref{GHP-ine}, for $t \geq q_{K+1} T_*$. 

\underline{Case (iii): $s =\frac{1}{2}$.} 
In this case, \eqref{u-m-upper-bound-est} reads 
\[  u(t,x) \leq C (T_*-t)^\frac{1}{1-m} \Phi (x) |\log \Phi(x)| \leq C (T_*-t)^\frac{1}{1-m} \Phi^{\frac{1-\eps}{m}} (x) \qquad \text{ for } t \in \left[\frac{T_*}{2},T_* \right], \]
for every $\eps > 0$. Since $m < 1$, we can choose $\eps$ so small that $\frac{1- \eps}{m} \geq 1$. Now we can argue as in the first part of Step 2 to deduce \eqref{GHP-ine} for every $t \geq \frac{2}{3} T_*$.

\textit{Step 3. The upper bound implies the lower bound. }

By \cite[Lemma 3.4]{BoIbIs}, we have for all $t \in [0,T_*] $ and a.e. $ x \in \Omega $
\begin{equation}\label{low-bound-prelim}
    \frac{u^m(t,x)}{t^\frac{m}{1-m}} \geq \frac{1}{1-m} \int_\Omega \frac{u(t,y)}{T_*^\frac{1}{1-m} - t^\frac{1}{1-m}} \mathbb G(x,y) \diff{y}.
\end{equation}
Using that  $ \mathbb G (x,y) \geq c_0 \Phi(x) \Phi(y) $ by \eqref{Greens estimate}, this implies 
\begin{equation}
    \label{low-bound-prelim-2}
   u^m(t,x) \geq c \Phi(x) (T_* - t)^{-1} \int_\Omega u(t,y) \Phi(y) \diff{y} \qquad \text{ for all } t \in [\frac{T_*}{2}, T_*]. 
\end{equation}
for some $c > 0$ (which may change from line to line). To bound the remaining integral, we estimate, for all $ t \in [0,T_*] $,
\begin{align}
    \frac{\diff}{\diff{t}} \int_\Omega u(t,y) \Phi(y) \diff{y} &= -\lambda_1 \int_\Omega u^m(t,y) \Phi(y) \diff{y} \nonumber \\
    &\leq - \lambda_1 \frac{1}{ \| \frac{u^m(t)}{\Phi} \|_{L^\infty(\Omega)}} \frac{1}{\| u(t)\|_{L^\infty(\Omega)}^{1-m}} \int_\Omega u^{1+m}(t,y) \diff{y}. \label{diff-integ-est}
\end{align}
By the upper bound we proved in Step 2, we have 
\[ \left\|\frac{u^m(t)}{\Phi}\right\|_{L^\infty(\Omega)} \leq C (T_* - t)^\frac{m}{1-m},  \quad \|u(t)\|_{L^\infty(\Omega)} \leq C (T_* - t)^\frac{1}{1-m} \quad \text{ for all }  t \in (t_*, T_*]. 
\]
Moreover, $\|u(t)\|_{L^{1+m}(\Omega)} \geq c (T_* - t)^\frac{1}{1-m}$ by \cite[Proposition 4]{BoIbIs} for all $t \in [0, T_*]$. Inserting these estimates into \eqref{diff-integ-est}, we get    
\begin{equation}
    \frac{\diff}{\diff{t}} \int_\Omega u(t,y) \Phi(y) \diff{y} \leq - c (T_*-t)^\frac{m}{1-m} \qquad \text{ for all } t \in (t_*,T_*].
\end{equation}
Integration over $[t, T_*]$ yields 
\begin{equation}
    \int_\Omega u(t,y) \Phi(y) \diff{y} \geq c (T_*-t)^\frac{1}{1-m} \qquad \text{ for all } t \in (t_*,T_*].
\end{equation}
Inserting this inequality into \eqref{low-bound-prelim-2}, we deduce 
\begin{equation}
    u^m(t,x) \geq c \Phi(x) (T_*-t)^\frac{m}{1-m} \qquad \text{ for all } t \in (t_*,T_*].
\end{equation}
This is the desired lower bound from \eqref{GHP-ine}. 
\end{proof}

\subsection{Proof of Theorem \ref{theorem uniform_conv_rel_err}}

In this section, we complete the proof of Theorem \ref{theorem uniform_conv_rel_err}. We obtain Theorem \ref{theorem uniform_conv_rel_err} from the following energy estimate in the spirit of \cite{Bonforte2021}. 

\begin{proposition}\label{proposition weighted smoothing effects-bdd}
    For $w$ a solution to \eqref{w equation-bdd} and $\varphi$ a solution to \eqref{stationary equation-bdd}, let $\mathrm{h}(\tau, x)  = \frac{w(\tau,x)}{\varphi(x)} - 1$ be the relative error function as defined in \eqref{relative-error-def}. Then there exists $C >0$ such that for any $\tau$ large enough,
    \begin{equation}
        \| \mathrm{h}(\tau, \cdot) \|_{L^\infty(\Omega)} \leq C \left( \sup_{\sigma \in [\tau - 1, \infty)} \|w(\sigma) - \varphi\|_{H(\Omega)} \right)^\frac{s}{N+\gamma},
    \end{equation}
    where $\gamma$ is defined by \eqref{gamma values}. 
\end{proposition}

\begin{proof}[Proof of Theorem \ref{theorem uniform_conv_rel_err}]
Theorem \ref{theorem uniform_conv_rel_err} is now an immediate consequence of the bound from Proposition \ref{proposition weighted smoothing effects-bdd}. 
\end{proof}

It remains to give the proof of Proposition \ref{proposition weighted smoothing effects-bdd}. 
As mentioned in the introduction, we will prove Proposition \ref{proposition weighted smoothing effects-bdd} based on ideas used in \cite[Theorem 4.1]{Bonforte2021}. To carry these out without the $L^\infty$-smallness assumption present in \cite{Bonforte2021}, we appeal to the Global Harnack Principle from Proposition \ref{GHP}.

Following \cite{Bonforte2021}, the proof of Proposition \ref{proposition weighted smoothing effects-bdd}, which we give at the end of this section, relies on the following two lemmas.

\begin{lemma}\label{lemma time_mono_est}
For $w$ a solution to \eqref{w equation-bdd} and $\varphi$ a solution to \eqref{stationary equation-bdd}, let $\mathrm{h}(\tau, x) = \frac{w(\tau,x)}{\varphi(x)} - 1$ be the relative error function as defined in \eqref{relative-error-def}. Then the following estimates hold true for any $ \tau_1 \geq \tau_0 \geq \frac{p}{p-1} \log 2$ and almost every $ x \in \Omega$:
\begin{equation}
    \begin{aligned}
        \frac{1}{2m} & \left( 1 - e^{-2m (\tau_1 - \tau_0)} \right) \mathrm{h}(\tau_1,x) - m (\tau_1 - \tau_0)^2 \\
        &\leq \int_{\tau_0}^{\tau_1} \mathrm{h}(\tau,x) \diff{\tau} \\
        &\leq  \frac{1}{2m}  \left(e^{2m (\tau_1 - \tau_0)} -1 \right) \mathrm{h}(\tau_1,x) + m (\tau_1 - \tau_0)^2 e^{2m (\tau_1 - \tau_0) }.
    \end{aligned}
\end{equation}
\end{lemma}

\begin{proof}
The celebrated Benilan-Crandall estimate \cite{BeCr} (see \cite[Lemma 5.4]{BoIbIs} for the fractional setting) says that for nonnegative solutions $u$ to \eqref{fde}, one has
$$
u_t \leq \frac{u}{(1-m) t}.
$$
For the rescaled solution $ w( \tau , x ) $ defined in \eqref{w in terms of u bdd}, let $ v( \tau , x ) : = w^\frac{1}{m}$, i.e.,
$$
v(\tau,x) = \left( \frac{p-1}{p} \right)^{-\frac{p}{p-1}} ( T_* - t )^{-\frac{p}{p-1}} u(t,x), \quad t = T_* \left( 1 - \exp{\left( -\frac{p-1}{p} \tau \right)} \right).
$$
By calculation, the Benilan-Crandall inequality for $v$ becomes
$$
\frac{v_\tau (\tau,x)}{v(\tau,x)} \leq \frac{1}{1 - \exp \left( - \frac{p-1}{p} \tau \right) } \leq 2,
$$
where in the last inequality we used that $ \tau \geq \frac{p}{p-1} \log 2 $.

Then, the above inequality implies 
\begin{equation*}
    \partial_\tau \mathrm{h} \leq 2m ( \mathrm{h} + 1 ),
\end{equation*}
and thus
$$
\mathrm{h}(\tau) + 1 \leq \left( \mathrm{h} ( \tau_* ) + 1 \right) e^{2m ( \tau - \tau_* )}.
$$
Hence, for all $ \tau \geq \tau_* \geq \frac{p}{p-1} \log 2 $,
$$
\mathrm{h}(\tau) \leq e^{2m (\tau- \tau_*)} \mathrm{h}(\tau_*) + 2m ( \tau - \tau_* ) e^{2m ( \tau - \tau_* )}.
$$
Similarly, for all $ \tau \geq \tau_* \geq \frac{p}{p-1} \log 2$,
$$
\mathrm{h} (\tau_*) \geq e^{-2m ( \tau - \tau_* )} \mathrm{h} (\tau) - ( \tau - \tau_* ).
$$
As a result, for all $ \tau \in [\tau_0 , \tau_1] \subset [ \frac{p}{p-1} \log 2 , \infty ) $ we obtain
$$
\frac{\mathrm{h} ( \tau_1)}{e^{2m ( \tau_1 - \tau )}} - 2m ( \tau_1 - \tau ) \leq \mathrm{h}(\tau) \leq e^{2m ( \tau - \tau_0 )} \mathrm{h(\tau_0)} + 2m ( \tau- \tau_0) e^{2m ( \tau - \tau_0 )}.
$$
An integration on $[\tau_0,\tau_1]$ yields the lemma.
\end{proof}

\begin{lemma}\label{lemma fun_poi_ine}
  For $w$ a solution to \eqref{w equation-bdd} and $\varphi$ a solution to \eqref{stationary equation-bdd}, let $\mathrm{h}(\tau, x) = \frac{w(\tau,x)}{\varphi(x)} - 1$ be the relative error function as defined in \eqref{relative-error-def}.  Then there are $M_1, M_2 > 0$ such that the following estimates hold true for any $ \tau_1 \geq \tau_0 \geq \frac{p}{p-1} \log 2$ and almost every $ x \in \Omega$:
    \begin{equation}
        \left| \int_{\tau_0}^{\tau_1} \mathrm{h}(\tau,x) \diff{\tau} \right| \leq [ M_1 + M_2 ( \tau_1 - \tau_0 ) ] \left( \sup_{\tau \in [\tau_0 , \tau_1] } \|w(\tau) - \varphi\|_{H(\Omega)}  \right)^\frac{2s}{N+\gamma},
    \end{equation}
    where $\gamma$ is given by \eqref{gamma values}. 
\end{lemma}

We follow the proof strategy of \cite[Lemma 4.4]{Bonforte2021}. However, due to the fractional setting and the various values which $\gamma$ can take, we need to refine the estimates of \cite{Bonforte2021} in some places. Indeed, the arguments from \cite{Bonforte2021} fail to give a useful estimate in some of the cases we consider, namely when $\gamma = 1$ (SFL) and $s \in (0,\frac{1}{2}]$. As a byproduct of this refinement, the decay exponent $\frac{2s}{N+\gamma}$ we obtain is better  than the exponent $\frac{2s-\gamma}{N}$ (which would result from following the technique of \cite{Bonforte2021}) in \emph{all} cases we consider.

\begin{proof}
    Using the equations for $w$ and $\varphi$, we find that $ \mathrm H(\tau,x) := w(\tau,x) - \varphi(x) $ satisfies
    \begin{equation}
        \label{H equation}
        (-\Delta)^s \mathrm H = -\partial_\tau w^p + (w^p - \varphi^p).  
    \end{equation} 
Since $u$ is a strong solution (see Definition \ref{weak-sol-def}), the right side of \eqref{H equation} is in $L^1(\Omega)$. Thus we can apply the Green's function for $(-\Delta)^s$ to invert \eqref{H equation} and get 
    \begin{equation*}
    \begin{aligned}
        \mathrm H(\tau,x) = - \int_\Omega (\partial_\tau w(\tau,y)^p) \mathbb{G} (x,y) \diff y + \int_\Omega (w(\tau,y)^p - \varphi(y)^p) \mathbb{G} (x,y) \diff y. 
        \end{aligned}
    \end{equation*}
    Integrating over $(\tau_0 , \tau_1)$ and recalling that $\mathrm H(\tau, x) = \mathrm h (\tau, x) \varphi(x)$, we get
    \begin{equation}\label{integral-H}
    \begin{aligned}
       \varphi(x) \int_{\tau_0}^{\tau_1} \mathrm h(\tau,x) \diff \tau&= \int_\Omega [w^p (\tau_0,y) - w^p (\tau_1,y) ] \mathbb{G} (x,y) \diff y \\
        &+ \int_{\tau_0}^{\tau_1} \int_\Omega (w^p(\tau,x) - \varphi^p(x)) \mathbb{G} (x,y) \diff y \diff \tau=: I_1 + I_2.
    \end{aligned}
    \end{equation}
    We now estimate $I_1$ and $I_2$ separately. 

   For $I_1$, with $r>0$ to be fixed, we split
    \begin{equation}
    \label{I1 split}
        \begin{aligned}
            | I_1 | &\leq \int_{B_r(x)} |w^p (\tau_0,y) - w^p (\tau_1,y) | \mathbb{G} (x,y) \diff y + \int_{\Omega \backslash B_r(x)} |w^p (\tau_0,y) - w^p (\tau_1,y) | \mathbb{G} (x,y) \diff y. 
        \end{aligned}
    \end{equation}

    On $\Omega \setminus B_r(x)$, by the numerical inequality $|x^p - y^p| \leq p (x^{p-1} \lor y^{p-1}) |x-y|$ (valid for all $x,y \geq 0$ and $ p \geq 1$) and again Corollary \ref{corollary harnack w}, we estimate 
    \begin{align*}
         | w^p(\tau_0,y) - w^p(\tau_1,y) | &\leq p \tilde{C_1}^{p-1} \| \varphi \|^{p-1}_{L^\infty(\Omega)} | w(\tau_0,y) - w(\tau_1,y) | \\
        &= p \tilde{C_1}^{p-1} \| \varphi \|^{p-1}_{L^\infty(\Omega)} | \mathrm H(\tau_0,y) - \mathrm H(\tau_1,y) | 
    \end{align*}
    where $\tilde{C_1}$ is the constant from Corollary \ref{corollary harnack w}. Moreover, we bound $\mathbb G(x,y) \leq C \frac{\varphi(x)}{|x-y|^{N-2s+\gamma}}$. This bound is a consequence of \eqref{Greens estimate} together with the fact that $\varphi \sim \Phi$ by Corollary \ref{corollary harnack w}. We obtain 
    \begin{align*}
       &\quad \int_{\Omega \backslash B_r(x)} |w^p (\tau_0,y) - w^p (\tau_1,y) | \mathbb{G} (x,y) \diff y  \\
       &\lesssim \varphi(x) \int_{\Omega \setminus B_r(x)} | \mathrm H(\tau_0,y) - \mathrm H(\tau_1,y)| \frac{1}{|x-y|^{N-2s+\gamma}} \diff y \\
       &\lesssim \varphi(x)  r^{-N+2s - \gamma} \|\mathrm H(\tau_0, \cdot) - \mathrm H(\tau_1, \cdot)\|_{L^1(\Omega)}  \\
              &\lesssim \varphi(x) 
 r^{-N+2s - \gamma} \|\mathrm H(\tau_0, \cdot) - \mathrm H(\tau_1, \cdot)\|_{H(\Omega)} \\
    &\lesssim \varphi(x)  r^{-N+2s - \gamma} \sup_{\tau \in [\tau_0, \tau_1]} \|\mathrm H(\tau, \cdot)\|_{H(\Omega)}. 
    \end{align*}
    The second-to-last inequality follows by applying Hölder's and Sobolev's inequalities.

On $B_r(x)$, we need to use the finer estimate $\mathbb G(x,y) \leq C \frac{1}{|x-y|^{N-2s}} \left( 
 \frac{\varphi(x)}{|x-y|^{\gamma}} \land 1 \right)$, which also follows from \eqref{Greens estimate} (notice that the cruder upper bound $\frac{1}{|x-y|^{N-2s+\gamma}}$ might not be integrable near $x$, namely if $\gamma = 1$ and $s \leq \frac{1}{2}$). Also, we will use the fact that there is $C > 0$ such that
 \begin{equation}
     \label{hölder varphi}
     |\varphi(x) - \varphi(y)| \leq C|x-y|^\gamma \qquad \text{ for every } x,y \in \Omega 
 \end{equation} 
 by the global $C^\gamma$-regularity of $\varphi$ (see \cite{RoSe2014} for RFL, \cite{CaSt2016} for SFL and \cite{Fall2022} for CFL). 
 
 If $r \leq \varphi(x)^\frac{1}{\gamma}$, we get (using Corollary \ref{corollary harnack w})
 \begin{align*}
     \int_{B_r(x)} |w^p (\tau_0,y) - w^p (\tau_1,y) | \mathbb{G} (x,y) \diff y &\lesssim \|\varphi\|_{L^\infty(B_{r}(x))}^p \int_{B_r(x)} \frac{1}{|x-y|^{N-2s}} \diff y \\
     & \lesssim (\varphi(x) + r^\gamma)^p r^{2s} \leq \varphi(x)^p r^{2s} \lesssim \varphi(x) r^{2s}. 
 \end{align*}
If $r > \varphi(x)^\frac{1}{\gamma}$, we split 
\begin{align*}
    &\quad  \int_{B_r(x)} |w^p (\tau_0,y) - w^p (\tau_1,y) | \mathbb{G} (x,y) \diff y \\
    &=   \int_{B_{\varphi(x)^{1/\gamma}}(x)} |w^p (\tau_0,y) - w^p (\tau_1,y) | \mathbb{G} (x,y) \diff y \\
    & \qquad \quad +  \int_{B_r(x) \setminus B_{\varphi(x)^{1/\gamma}}(x)} |w^p (\tau_0,y) - w^p (\tau_1,y) | \mathbb{G} (x,y) \diff y. 
\end{align*}
For the first term, we get as before 
\[ \int_{B_{\varphi(x)^{1/\gamma}}(x)} |w^p (\tau_0,y) - w^p (\tau_1,y) | \mathbb{G} (x,y) \diff y \lesssim \varphi^{p + \frac{2s}{\gamma}} \leq \varphi(x) r^{2s + \gamma(p-1)}.   \]
For the second term, using in particular the Hölder estimate \eqref{hölder varphi} on $\varphi$,
\begin{align*}
     &\qquad \int_{B_r(x) \setminus B_{\varphi(x)^{1/\gamma}}(x)} |w^p (\tau_0,y) - w^p (\tau_1,y) | \mathbb{G} (x,y) \diff y \\
     &\lesssim \varphi(x) \int_{B_r(x) \setminus B_{\varphi(x)^{1/\gamma}}(x)} \varphi^p(y) \frac{1}{|x-y|^{N-2s + \gamma}} \diff y  \\
     & \lesssim \varphi(x) \int_{B_r(x) \setminus B_{\varphi(x)^{1/\gamma}}(x)} (\varphi(x)+ |x-y|^\gamma)^p \frac{1}{|x-y|^{N-2s + \gamma}} \diff y  \\
     &\lesssim \varphi(x) \int_{B_r(x)} |x-y|^{-N + 2s + \gamma(p-1)} \diff y \lesssim \varphi(x) r^{2s + \gamma(p-1)}. 
\end{align*}
Collecting everything, we have proved that 
\begin{equation}
    \label{I1 bound}
    I_1 \lesssim \varphi(x) \left( r^{2s} + r^{2s + \gamma(p-1)} +  r^{-N+2s - \gamma} M \right) \lesssim \varphi(x) \left( r^{2s} +  r^{-N+2s - \gamma} M \right), 
\end{equation} 
with $M = \sup_{\tau \in [\tau_0, \tau_1]} \|\mathrm H(\tau, \cdot)\|_{H(\Omega)}$. Now choosing $r = M^\frac{1}{N+\gamma}$ gives 
\[ I_1 \lesssim  \varphi(x) M^{\frac{2s}{N+\gamma}}. \]

 The estimates for $I_2$ are almost identical to those for $I_1$, up to an additional integration over $[\tau_0, \tau_1]$. Indeed, arguing as for $I_1$, we can estimate
    \begin{equation}
        \begin{aligned}
             \int_\Omega |w^p(\tau,y) - \varphi^p| \mathbb G(x,y) \diff y &\leq C \varphi(x)  \left( r^{2s}  +  r^{-N+2s - \gamma} M \right). 
        \end{aligned}
    \end{equation}
Integrating over $[\tau_0, \tau_1]$ and choosing again $r = M^\frac{1}{N+\gamma}$, we get  
\begin{equation}
    \label{I2 bound}
    I_2\lesssim (\tau_1 - \tau_0) \varphi(x) \left( M^{\frac{2s}{N+\gamma}} + M^\frac{2s + \gamma(p-1)}{N+\gamma} \right) \lesssim  (\tau_1 - \tau_0) \varphi(x) M^{\frac{2s}{N+\gamma}}. 
\end{equation} 
    Then the lemma follows combining \eqref{integral-H} with \eqref{I1 bound} and \eqref{I2 bound}.
\end{proof}

At last we can give the proof of Proposition \ref{proposition weighted smoothing effects-bdd}, the main result in this section. 
\begin{proof}
    [Proof of Proposition \ref{proposition weighted smoothing effects-bdd}]
    We firstly prove that 
        \begin{equation}
        \label{linfty bound two terms}
        \| \mathrm{h}(\tau) \|_{L^\infty(\Omega)} \leq C \frac{e^{2m (\tau-\tau_0)}}{\tau-\tau_0} \left( \sup_{ s \in [ \tau_0 , \tau ]} \|w(\tau) - \varphi\|_{H(\Omega)} \right)^\frac{2s}{N+\gamma} + 2m ( \tau - \tau_0 ) e^{2m (\tau - \tau_0)}.
    \end{equation}
Indeed, by combining the first inequality of Lemma \ref{lemma time_mono_est} with Lemma \ref{lemma fun_poi_ine}, we obtain, for every $\tau_1 \geq \tau_0 \geq \frac{p}{p-1} \ln 2$, and almost every $x \in \Omega$, 
\begin{align*}
   \mathrm h(\tau_1,x) &\leq \frac{2m}{1 - e^{-2m (\tau_1 - \tau_0)}} \left((M_1 + M_2(\tau_1 - \tau_0) \sup_{s \in [\tau_0, \tau_1]} \|w(\tau) - \varphi\|_{H(\Omega)}^\frac{2s}{N + \gamma} + m (\tau_1 - \tau_0)^2 \right). 
\end{align*}
Since $1 - e^{-2m (\tau_1 - \tau_0)} \geq 2m (\tau_1 - \tau_0) e^{-2m (\tau_1 - \tau_0)}$ by convexity of $t \mapsto e^{-2m t}$, we obtain
\begin{align*}
   \mathrm h(\tau_1,x) &\leq M_1 \frac{e^{2m (\tau_1 - \tau_0)}}{\tau_1 - \tau_0} \sup_{s \in [\tau_0, \tau_1]} \|w(\tau) - \varphi\|_{H(\Omega)}^\frac{2s}{N+\gamma} + M_2 e^{2m (\tau_1 - \tau_0)} \sup_{s \in [\tau_0, \tau_1]} \|w(\tau) - \varphi\|_{H(\Omega)}^\frac{2s}{N+\gamma} \\
    & \qquad + m e^{2m(\tau_1 - \tau_0)} (\tau_1 - \tau_0). 
\end{align*}
For $\tau_0$ large enough, we have $M_2 e^{2m (\tau_1 - \tau_0)} \sup_{s \in [\tau_0, \tau_1]} \|w(\tau) - \varphi\|_{H(\Omega)}^\frac{2s}{N+ \gamma} \leq m$. Then the middle term  can be absorbed into the other two, and we obtain
\[\mathrm h(\tau_1, x) \leq (M_1 + m) \frac{e^{2m (\tau_1 - \tau_0)}}{\tau_1 - \tau_0} \sup_{s \in [\tau_0, \tau_1]} \|w(\tau) - \varphi\|_{H(\Omega)}^\frac{2s}{N+\gamma} + 2 m e^{2m(\tau_1 - \tau_0)} (\tau_1 - \tau_0).    \]

Similarly, combining the second inequality of Lemma \ref{lemma time_mono_est} with Lemma \ref{lemma fun_poi_ine}, we obtain, for every $\tau_1 \geq \tau_0 \geq \frac{p}{p-1} \ln 2$, and almost every $x \in \Omega$, 
\begin{align*}
    -\mathrm h(\tau_1, x) \leq \frac{2m}{e^{2m(\tau_1 - \tau_0)}} \left((M_1 + M_2 (\tau_1 - \tau_0)) \sup_{s \in [\tau_0, \tau_1]} \|w(\tau) - \varphi\|_{H(\Omega)}^\frac{2s}{N+\gamma}  + m (\tau_1 - \tau_0)^2  \right). 
\end{align*}
Since $e^{2m (\tau_1 - \tau_0)} -1 \geq 2m (\tau_1 - \tau_0) e^{2m (\tau_1 - \tau_0)} \geq 2m (\tau_1 - \tau_0)$ by convexity of $t \mapsto e^{2m t}$, we obtain
\begin{align*}
     -\mathrm h(\tau_1, x) &\leq \left(\frac{M_1}{\tau_1 - \tau_0} + M_2 \right) \sup_{s \in [\tau_0, \tau_1]} \|w(\tau) - \varphi\|_{H(\Omega)}^\frac{2s}{N+\gamma} +   m (\tau_1 - \tau_0) e^{2m (\tau_1 - \tau_0)} \\
     & \leq \left(M_1 + \frac{M_2}{2m}\right) \frac{e^{2m(\tau_1 - \tau_0)}}{\tau_1 - \tau_0} \sup_{s \in [\tau_0, \tau_1]} \|w(\tau) - \varphi\|_{H(\Omega)}^\frac{2s}{N+\gamma} +   m (\tau_1 - \tau_0) e^{2m (\tau_1 - \tau_0)}.
\end{align*}
The bound \eqref{linfty bound two terms} now follows with $\tau = \tau_1$ by combining the obtained upper bounds on $\pm \mathrm h(\tau_1, x)$. 

It is now not difficult to deduce the proposition from \eqref{linfty bound two terms}. Indeed, for given $\tau > 0$, set 
\[ \tau_0 := \tau - \sup_{\sigma \in [\tau-1, \infty)} \left( \|w(\sigma) - \varphi\|_{H(\Omega)} \right)^\frac{s}{N+\gamma}. \]
Since $w(\tau) \to \varphi$ in $H(\Omega)$, we have $\tau - \tau_0 < 1$ for every large enough $\tau$. Then $e^{2m (\tau - \tau_0)} < 1$, and \eqref{linfty bound two terms} yields
\begin{align*}
     \| \mathrm{h}(\tau) \|_{L^\infty(\Omega)} &\leq C \frac{e^{2m (\tau-\tau_0)}}{\tau-\tau_0} \left( \sup_{ s \in [ \tau_0 , \tau ]} \|w(\tau) - \varphi\|_{H(\Omega)} \right)^\frac{2s}{N+\gamma} + 2m ( \tau - \tau_0 ) e^{2m (\tau - \tau_0)} \\
     &= C \frac{1}{\tau - \tau_0} (\tau - \tau_0)^2 + 2m (\tau - \tau_0)  = (C + 2m) (\tau - \tau_0) \\
     &= (C + 2m) \sup_{\sigma \in [\tau-1, \infty)} \left( \|w(\sigma) - \varphi\|_{H(\Omega)} \right)^\frac{s}{N+\gamma}.  
\end{align*}
This completes the proof. 
\end{proof}

\subsection{Proof of Theorem \ref{thm-nonnegative-bdd}}

The proof of Theorem \ref{thm-nonnegative-bdd} is similar to, but simpler than the proof of Theorem \ref{theorem exponential decay}, so we will only point out the necessary adaptations. 

To start with, the following statement replaces Proposition \ref{proposition linearized operator}. We recall its standard proof for the convenience of the reader. 

\begin{lemma}
    \label{lemma eigenvalues bounded}
    Let $p \in (1, \frac{N+2s}{N-2s})$. Let $\varphi$ be a positive solution to \eqref{stationary equation-bdd} and let $\mathcal L_\varphi = (-\Delta)^s - p \varphi^{p-1}$. Then there is an increasing sequence of eigenvalues $\nu_j \to \infty$ and an associated sequence of eigenfunctions $(e_j)_{j \in \mathbb N_0} \subset H(\Omega)$ satisfying 
    \[ \mathcal L_\varphi e_j = \nu_j \varphi^{p-1} e_j, \]
    which can be taken to satisfy $q(e_i, e_j) = \delta_{ij}$. 
\end{lemma}

Here $q(\cdot, \cdot)$ denotes the quadratic form of $(-\Delta)^s$ from Section \ref{subsubsec laplacians}. By Poincaré's inequality, $q(\cdot, \cdot)$ is a scalar product  on $H(\Omega)$ equivalent to the standard one. 

\begin{proof}
It suffices to show that $ T = T_\varphi := \frac{(-\Delta)^s}{\varphi^{p-1}}$ is the inverse of a compact, self-adjoint operator on $L^2(\Omega, \varphi^{p-1} \diff x)$. This operator is constructed as the composition
\[ L^2(\Omega, \varphi^{p-1} \diff x) \longrightarrow H(\Omega) \longrightarrow L^2(\Omega) \longrightarrow L^2(\Omega, \varphi^{p-1} \diff x). \]
Here, the second arrow is the compact embedding $H(\Omega) \to L^2(\Omega)$. The embedding $L^2(\Omega) \to L^2(\Omega, \varphi^{p-1} \diff x)$ is continuous because $\varphi$ is bounded. The first operator is the solution operator mapping $f \in L^2(\Omega, \varphi^{p-1} \diff x)$ to a (weak) solution $u$ of $(-\Delta)^s u = f$. It is well-defined and continuous by Riesz's representation theorem and the continuous embedding $H(\Omega) \to L^2(\Omega, \varphi^{p-1} \diff x)$. Since the second map is compact and all three are continuous, the composition is compact and maps $f$ to $u$ such that $(-\Delta)^s u = f$. This is the weak formulation of $T u = f$ on $L^2(\Omega, \varphi^{p-1} \diff x)$, so the map we constructed is the desired inverse of $T$.
\end{proof}

At this point, to prove Theorem \ref{thm-nonnegative-bdd}, we can carry out the procedure used in Section \ref{section proof of theorem exp decay RN} to prove Theorem \ref{theorem exponential decay}. Since the proof is largely identical (and in fact simpler because of the non-degeneracy assumption), we omit the details and only briefly point out the relevant changes in the following paragraph. 

Clearly, the objects $\R^N$, $U$, $\dhs$ and $\dot H^{-s}(\R^N)$ from Section 3 need to be replaced by $\Omega$,  $\varphi$, $H(\Omega)$ and $H^*(\Omega)$ respectively. Testing with $\partial_\tau w$ in Lemma 3.1 can be made rigorous as explained in \cite[p.8]{Akagi2023} and references therein. Since by assumption, the kernel of $\mathcal L_\varphi$ is trivial, Step 1 in the proof of Lemma \ref{lemma hard ineq} is not needed and we can work directly with $\varphi$ instead of $U(\tau)$ in the remaining proof of Lemma \ref{lemma hard ineq}. The role of the eigenvalue $\nu_{N+2}$ in that proof is now naturally played by the eigenvalue $\tilde \nu$, each of them representing the smallest positive eigenvalue. Since in the bounded domain case, $\mathcal L_\varphi$ may a priori have more than one negative eigenvalue, instead of the decomposition $\rho(\tau) = \sigma_0(\tau) e_0(\tau) + \Tilde{\rho}(\tau)$ from \eqref{rho decomp}, we need to write more generally 
\[ \rho(\tau) = \sum_{j: \nu_j < 0} \sigma_j(\tau) e_j + \Tilde{\rho}, \]
with $\Tilde{\rho} = \sum_{j: \nu_j > 0} \sigma_j(\tau) e_j$; compare \eqref{Tilde rho}. Replacing the quantity $\sigma_0^2(\tau)$ by $\sum_{j: \nu_j < 0} \sigma_j(\tau)^2$, all the estimates based on \eqref{rho decomp} in the proof of Lemmas \ref{lemma hard ineq} and \ref{lemma rho L rho estimate} go through without further changes. Among the auxiliary lemmas from Section \ref{subsec auxiliary}, Lemmas \ref{Taylor_for} and \ref{J-ctrl-rho} can be proved without changes. Lemma \ref{rel_err_exponential} is replaced by Proposition \ref{proposition weighted smoothing effects-bdd}. Lemma \ref{lemma U z lambda} is not needed because of the non-degeneracy assumption on $\varphi$.  

\begin{remark}
    In fact, the merit of the proof of Theorem \ref{theorem exponential decay} given in Section \ref{section proof of theorem exp decay RN} with respect to the proof in Akagi's paper \cite{Akagi2023} is that it allows for degeneracy of the kernel as in \eqref{nondeg} and that it avoids the use of the inverse operator $\mathcal L_\varphi^{-1}$. However, in the situation of Theorem \ref{thm-nonnegative-bdd} these additional difficulties are not in place. Hence, as an alternative to the proof sketched above, once Theorem 
\ref{theorem uniform_conv_rel_err} is available, the proof of Theorem \ref{thm-nonnegative-bdd} could also be carried out by directly following the arguments from \cite{Akagi2023}. 
\end{remark}

\bibliographystyle{abbrv}
\bibliography{ffd-sharp-decay-ref.bib}

\vfill

\end{document}